\documentclass[12pt, letterpaper]{amsart}
\usepackage[utf8]{inputenc}
\usepackage{amssymb}
\usepackage{amsmath}
\usepackage{cite}
\usepackage{graphicx}
\usepackage{float}
\usepackage{hyperref}
\usepackage{cleveref}
\usepackage[margin=1in]{geometry}
\usepackage{mathtools}
\usepackage{mathrsfs}
\usepackage{xcolor}
\usepackage{lipsum}
\usepackage{mwe}
\usepackage{enumitem}

\newcommand{\CI}{\mathcal{C}^{\infty}}

\newcommand{\RR}{\mathbf{R}}

\DeclareMathOperator{\sgn}{sgn}

\newcommand{\restrictedto}{\big\rvert}

\newcommand{\Lap}{\Delta}

\newcommand{\RNum}[1]{\uppercase\expandafter{\romannumeral #1\relax}}

\newcommand{\df}{\coloneqq}
\newcommand{\la}{\langle}
\newcommand{\ra}{\rangle}

\newcommand{\RRS}{\RR^2\backslash S}

\newcommand{\WF}{\operatorname{WF}}
\newcommand{\aalpha}{\boldsymbol{\alpha}}
\newcommand{\pa}{P_{\alpha}}

\DeclareMathOperator{\im}{\text{Im}}
\DeclareMathOperator{\ima}{\emph{Im}}
\DeclareMathOperator{\re}{\text{Re}}
\DeclareMathOperator{\rea}{\emph{Re}}
\DeclareMathOperator{\Tr}{Tr}

\DeclareMathOperator{\supp}{supp}

\DeclareMathOperator{\Id}{Id}

\newcommand{\dom}{\mathcal{D}}

\newcommand{\ssubset}{\subset\joinrel\subset}

\newtheorem{theorem}{Theorem}[section]
\newtheorem*{theorem*}{Theorem}
\newtheorem*{corollary*}{Corollary}
\newtheorem{corollary}[theorem]{Corollary}
\newtheorem{lemma}[theorem]{Lemma}
\newtheorem{definition}[theorem]{Definition}
\newtheorem{proposition}[theorem]{Proposition}
\newtheorem{remark}[theorem]{Remark}
\newtheorem*{property1}{Partition Property 1}

\title[Aharonov--Bohm Trace]{The wave trace and resonances of the magnetic Hamiltonian with singular vector potentials}
\author{Mengxuan Yang}
\date{\today}
\address{Department of Mathematics, Northwestern University}
\email{mxyang@math.northwestern.edu}

\begin{document}

\maketitle
\begin{abstract}
    We study leading order singularities of the wave trace of the Aharonov--Bohm Hamiltonian on $\RR^2$ with multiple solenoids under a generic assumption that no three solenoids are collinear. Then we apply our formula to get a lower bound of scattering resonances in a logarithmic neighborhood near the real axis.
\end{abstract}

\section{Introduction}
We study the singularities of the wave trace and resonances of the half-wave propagator $U_{\aalpha}(t)=e^{it\sqrt{P_{\aalpha}}}$ of the electromagnetic Hamiltonian:
\begin{equation*}
    P_{\aalpha}=\left(\frac{1}{i}\nabla-\Vec{A}\right)^2\ \text{on }\RRS
\end{equation*}
where $S=\{s_i=(x_i,y_i)|1\leq i\leq n\}$ and $\Vec{A}$ is the sum of $n$ singular vector potentials defined in equation \eqref{potential}. This  corresponds to the Aharonov--Bohm Hamiltonian \cite{aharonov1959significance} of $n$ infinitely thin solenoids in $\RR^3$ parallel to the $z$-axis. In particular, while the magnetic potential is non-vanishing everywhere, the magnetic field $\Vec{B}=\nabla\times \Vec{A}$ vanishes in $\RRS$, so the vector potential $\Vec{A}$ is curl-free.

The main result of this paper describes the singularities of the wave trace of $U_{\aalpha}(t)$ under a generic assumption that \emph{no three solenoids are collinear}:

\begin{theorem}
\label{thm1.1}
Consider the wave propagator $U_{\aalpha}(t)=e^{it\sqrt{P_{\aalpha}}}$ for the Hamiltonian $P_{\aalpha}$. The regularized wave propagator $U_{\aalpha}(t)-U_0(t)$ defined in Section \ref{reg&part} is in the trace class in the distributional sense. The singularities of the regularized trace $\Tr \big[ U_{\aalpha}(t)-U_{0}(t)\big]$ for $t>0$ are given by lengths of all closed polygonal trajectories in $\RRS$, with vertices at $S$. Moreover, the contribution to the singularity at $t=L$ from the closed polygonal trajectory $\gamma$ with length $L$ is given by the oscillatory integral expression:
\begin{equation*}
\int_{\RR_{\lambda}} e^{i\lambda\left(t-L \right)} a(t;\lambda)d\lambda     
\end{equation*}
where the amplitude $a\in S^{-\frac{m}{2}}(\RR; \RR_{\lambda})$ is given by 
\begin{equation*}
    (2\pi)^{\frac{m-2}{2}} \cdot e^{-i\frac{m\pi}{4}} \cdot L_0 \cdot \frac{d_{\gamma,\aalpha}}{{\prod_{i=1}^{m}l_i}^{\frac{1}{2}}} \cdot \rho(\lambda) \cdot \lambda^{-\frac{m}{2}} \mod S^{-\frac{m+1}{2}}
\end{equation*}
where $m$ stands for the number of corners of the polygonal path $\gamma$ (with multiplicities in case $\gamma$ loops itself); $l_i$ is the length of the $i$-th segment of $\gamma$; $L_0$ is the primitive length of $\gamma$ in case it loops itself more than once; $\rho(\lambda)\in \CI(\RR_{\lambda})$ is a smooth function satisfying $\rho\equiv 0$ for $\lambda<0$ and $\rho\equiv 1$ for $\lambda>1$; the coefficient $d_{\gamma,\aalpha}$ is given by
$$\left(\prod_{l=1}^{m} \sin(\pi\alpha_{k_l})\cdot \frac{e^{-i\left(\frac{\beta_{l}}{2}\right)}}{\cos(\frac{\beta_{l}}{2})}\right)
        \cdot \left(\prod_{k=1}^{n} e^{-2\pi i\cdot\alpha_k\cdot w_{\gamma,s_k}} \right)
$$
where the first term depends on the magnetic flux $\alpha_{k_l}$ and the angle $\beta_{l}$ at the $l$-th vertex along $\gamma$; the second term depends on the fractional winding number $w_{\gamma,s_k}$ (see Definition \ref{frac_winding}) of $\gamma$ with respect to the solenoid $s_k$ and the magnetic flux $\alpha_k$ there.
\end{theorem}
 
\begin{remark}
The first term in the above formula is given by the product of the so-called diffraction coefficients (see Section \ref{single_d}), while the second term is the holonomy contribution from the closed polygonal trajectory $\gamma$ due to the presence of the vector potentials.
\end{remark}

An application of our main theorem yields a lower bound on the number of scattering resonances within certain logarithmic neighborhood of the positive real axis:
\begin{corollary*}
Let $\mu_j$ be resonances of Hamiltonian $P_{\aalpha}$ on $\RRS$ subject to the geometric assumptions of the previous theorem. Then for any $\delta>0$, there exists a $r(\delta)>0$ such that
\begin{equation*}
    \#\left\{\mu_j\bigg\rvert \ 0\leq -\ima \mu_j\leq \left(\frac{1}{2d_{\text{max}}}+\epsilon\right)\cdot\log |\mu_j| ,\ \rea\mu_j\leq r\right\}\geq r^{1-\delta},
\end{equation*}
if $r>r(\delta)$, where $d_{\text{max}}$ is the maximal distance among the distances between all solenoids.
\end{corollary*}

The presence of the solenoids with the singular vector potentials generates a diffraction effect. The diffraction refers to the effect that when a propagating wave or a quantum particle encounters a corner of an obstacle or a slit, its wave front bends around the corner of the obstacle and propagates into the geometrical shadow region. For the wave equation on a singular domain, the singularities of the wave equation likewise split into two types after they encounter the singularity of the domain. One propagates along the natural geometric extension of the incoming ray, while other singularities emerge at the singular point and start propagating along all outgoing directions as a spherical wave called the diffractive wave. The singular vector potential here likewise generates a diffractive wave as in the situation of a singular domain. 

In order to prove the main theorem, by breaking down the propagation time $t$ and inserting microlocal cutoffs, we first compute a microlocalized diffractive propagator
$$\Pi_e U_{\aalpha,\gamma}(t) \Pi_s \df \Pi_e U_{\aalpha}(t-T_{n-1})\Pi_{n-1}\cdots\Pi_1 U_{\aalpha}(t_1) \Pi_s$$
of $P_{\aalpha}$ on $\RRS$ along a sequence of diffractions, where $\Pi_{\bullet}$ are pseudodifferential operators along certain diffractive polygonal trajectory $\gamma$. The computation of the diffractive propagator under a single diffraction
$$\tilde{\Pi}_{k+1} U_{\aalpha}(t_{k+1}) \tilde{\Pi}_{k}$$
is based on a previous result of the author \cite{yang2020ab} for one solenoid; we introduce additional gauge transformations in order to deal with the presence of multiple magnetic vector potentials. Then we use the theory of FIOs to obtain an oscillatory integral expression of the microlocalized propagator
$$\Pi_e U_{\aalpha,\gamma}(t) \Pi_s$$
under multiple diffractions.

To compute the singularities of the regularized trace near $t=L$ for some $L$, using a microlocal partition of unity, we decompose the (regularized) trace into a sum of finitely many microlocalized traces modulo smooth errors; each such microlocalized trace is the trace of the aforementioned microlocalized propagator:
\begin{equation*}
    \Tr \big[ U_{\aalpha}(t)-U_{0}(t)\big] \equiv  \sum_{1\leq k\leq M} \Tr \big[U_{\aalpha,\gamma}(t) B_{k}\big] \mod \CI
\end{equation*}
where $B_k\in\Psi^0_c(X)$ are a family of pseudodifferential operators.
Therefore, it remains to examine the contribution from each microlocalized trace using the oscillatory integral expression of the corresponding microlocalized propagator, then assemble all pieces using the microlocal partition of unity. The trace decomposition and standard propagation of singularities implies that the singularities of the regularized wave trace come from the microlocalized trace along closed diffractive geodesics, which are closed polygonal trajectories with vertices at the solenoids.

The wave trace and its singularities relate directly to the spectral properties of the underlying spaces, such as eigenvalues and scattering resonances. In the classic paper of Duistermaat--Guillemin \cite{duistermaat1975spectrum}, they showed the singularities of the wave trace of the Laplacian at the length of a closed geodesic on closed Riemannian manifold. Similar results were obtained by Hillairet \cite{hillairet2005contribution} on the flat surfaces with cone points and Ford--Wunsch \cite{ford2017diffractive} on general Riemanian manifold with conic singularities. Our work extends the series of projects into the framework of the Aharonov--Bohm Hamiltonian with singular vector potentials.  For the Poisson formula which relate the regularized trace to scattering resonances, we refer to \cite{bardos1982relation} and \cite{melrose1982scattering}. In this paper, we provide a regularization of the wave trace following the work of Sj\"ostrand \cite{sjostrand1997trace} which yields a local trace formula of scattering resonances. 

For the Aharonov--Bohm Hamiltonian, the singularities of the wave trace of certain closed geometric geodesics were studied by Eskin--Ralston \cite{eskin2014aharonov} in equilateral triangles with one solenoid using the technique of Gaussian beams and FIOs. They were able to recover the cosine of the flux using the leading singularities at a certain closed geometric geodesic. On the other hand, as we shall see later in this paper, we are able to recover the product of the sine of fluxes using the leading singularities of closed diffractive geodesics. Bogomolny--Pavloff--Schmit \cite{bogomolny2000diffractive} studied the diffractive singularities of rectangular billiards with one solenoid using the geometric theory of diffraction. They were only able to deal with the diffractive singularities of one solenoid in a rectangular domain.

There are various works on scattering theory of the Aharonov--Bohm Hamiltonian with multiple solenoids. {\v{S}}t'ov{\'\i}{\v{c}}ek\cite{stovicek1989green} studied the two solenoids case using a universal covering of the punctured plane, and finitely many solenoids in \cite{vstovivcek1991krein}; the results there are in terms of an infinite sum which is difficult to see the locations and amplitudes of the singularities for our purposes. Tamura\cite{tamura2007semiclassical} \cite{tamura2008time} and Ito--Tamura\cite{ito2001aharonov} \cite{ito2006semiclassical} studied the (semiclassical) scattering amplitudes and cross-sections of two solenoids with the distance going to infinity. Alexandrova--Tamura \cite{alexandrova2011resonance} \cite{alexandrova2014resonances} studied scattering resonances of two solenoids at large separation (the distance between two solenoids $d\rightarrow\infty$) using a modified complex scaling method. They also computed the resonances located in a logarithmic neighborhood for high energy. Our result generalizes the existence part of their result to arbitrary number of solenoids. The resonances generated by three and four solenoids with large separations were also studied by Tamura in \cite{tamura2017aharonov}. 

The novelty of this work is the following: this paper expand the Duistermaat--Guillemin trace formula \cite{duistermaat1975spectrum} to Hamiltonians with singular vector potentials; to the best of our knowledge, this is the first result on the trace formula for the Aharonov--Bohm Hamiltonian with multiple solenoids/vector potentials. Using our result, we also obtain a lower bound on the number of resonances within a logarithmic neighborhood of the real axis for the Aharonov--Bohm Hamiltonian. In particular, the advantage of the resonances bound in this paper is that unlike previous ones, our results applies to arbitrary number of solenoids with arbitrary distances between them.  

\subsection*{Acknowledgements:}
The author is greatly indebted to Jared Wunsch for many instructive discussions as well as valuable comments on the manuscript. The author would also like to thank Luc Hillairet for suggesting this interesting topic and providing helpful discussions and Maciej Zworski for valuable comments on the manuscript.
\section{Preliminaries}
In this section, we discuss some preliminaries on the operator $P_{\aalpha}$ and the diffractive geometry on $\RR^2$ with $n$ solenoids.

\subsection{Operators and domains}

We study the electromagnetic Hamiltonian 
\begin{equation}
\label{TotHam}
    P_{\aalpha}=\left(\frac{1}{i}\nabla-\Vec{A}\right)^2
\end{equation}
on the space $X\df\RRS$, where $S=\{s_i=(x_i,y_i)|1\leq i\leq n\}$ corresponds to locations of $n$ solenoinds and $\Vec{A}=\sum_{i=1}^{n}\Vec{A}_i$ with 
\begin{equation}
\label{potential}
    \Vec{A_i}=-\alpha_i\cdot\left(-\frac{y-y_i}{(x-x_i)^2+(y-y_i)^2}, \frac{x-x_i}{(x-x_i)^2+(y-y_i)^2} \right)
\end{equation}
being the $i$-th vector potential corresponding to $s_i$, and $\aalpha=(\alpha_1,\cdots,\alpha_n)$ is the multi-index of the magnetic fluxes of all $n$ solenoids. Without loss of generality, we assume $\alpha_i\in(0,1)$ for $1\leq i\leq n$. Note that the magnetic potential $\Vec{A}$ is singular at $s_i$ for $1\leq i\leq n$ and curl-free; therefore there is no magnetic field in $\RRS$. However, the motion of electrons in $\RRS$ can still ``feel" the influence of the magnetic field even though the electrons are completely shielded from the magnetic field. The electrons will experience a phase shift once we change the flux of the magnetic field, which can be observed by an interference experiment, although classically the change of magnetic flux has no influence on the motion of particles. This phenomenon generated by the singular magnetic potential $\Vec{A}$ is the so-called Aharonov--Bohm effect \cite{aharonov1959significance}, which suggests that the electromagnetic vector potential is more physical than the electromagnetic field in quantum mechanics.

Note that $P_{\aalpha}$ is a positive symmetric operator defined on $\CI_c(X)\subset L^2(\RR^2)$ with deficiency indices $(2n,2n)$. Therefore it admits various self-adjoint extensions. In particular, we choose the Friedrichs self-adjoint extension, which corresponds to the following function space:
$$\text{Dom}(P^{\text{Fr}}_{\aalpha})\df \dom^2_{\aalpha}=\left\{u\in L^2: P_{\aalpha} u\in L^2, u\rvert_S=0\right\}.$$
From now on, we use $P_{\aalpha}$ to denote the Friedrichs extension of the Aharonov--Bohm Hamiltonian. For detailed discussions of self-adjoint extensions of the Aharonov--Bohm Hamiltonian, we refer to \cite{adami1998aharonov} and \cite{stovicek1998aharonov}; for asymptotic behaviors of operator extensions of the Aharonov--Bohm Hamiltonian near the boundary (solenoids), see also \cite{stovicek2002generalized} \cite{mine2005aharonov} \cite{yang2020ab}. 

We define power domains by 
$$\dom_{\aalpha}^s\df \left\{u \in \dom_{\aalpha} : P^{s/2}_{\aalpha} u\in L^2\right\}$$
where $P^{s/2}_{\aalpha}$ is defined using the functional calculus. It is important to notice that away from the set of solenoids $S$, $\dom_{\aalpha}^s$ is agree with the usual Sobolev space $H^s$. 

Define the wave operator corresponding to this electromagnetic Hamiltonian:
\begin{equation*}
    \Box\df D_t^2-P_{\aalpha},
\end{equation*}
where $D_t=\frac{1}{i}\partial_t$. In particular, we want to study the (half-)wave propagator 
$$U_{\aalpha}(t)=e^{it\sqrt{P_{\aalpha}}}$$ 
and its trace $``\Tr U_{\aalpha}(t)"$. One crucial fact we need to point out is that even in the sense of distributions, $U_{\aalpha}(t)$ is not in the trace class. Therefore we need a regularization using a ``free" propagator  (cf. \cite{melrose1982scattering}). The long-range effect of the magnetic vector potential falls into the framework of \cite{sjostrand1997trace}, as long as we choose the corresponding ``free" propagator properly. We postpone the detailed discussion regarding this to Section \ref{reg&part}.

\subsection{Diffractive geometry}
\label{DG}
Now we give a brief introduction to the diffractive geometry on $X=\RRS$. This is indeed a simplified version of the diffractive geometry in \cite{ford2017diffractive} on Riemannian manifold with conic singularities. Note that all \emph{regular geodesics} in $X$ are simply straight line segments in $X=\RRS$. Therefore, by standard propagation of singularities \cite{duistermaat1972fourier}, away from the solenoid set $S$, singularities of the wave equation propagate along the straight lines in $\RRS$. However, near the solenoid there are two other types of \emph{generalized geodesics}, along which the singularities propagate, passing through the solenoids, which correspond to the diffractive and geometric waves emanating from the solenoid after the diffraction:

\begin{definition} 
Suppose $\gamma: [a,b]\rightarrow \RR^2$ is a continuous map whose image is a \emph{polygonal trajectory} with vertices at the solenoid set $S$.
\begin{itemize}
\item The polygonal trajectory $\gamma$ is a \emph{diffractive geodesic} if $\gamma^{-1}(S)$ is non-empty and $\gamma\big([a,b]\backslash\gamma^{-1}(S)\big)$ are finitely many straight line segments concatenated by the points in $S$.
\item The polygonal trajectory $\gamma$ is a (partially) \emph{geometric geodesic} if it is a diffractive geodesic and it contains a straight line segment parametrized by $[c,d]\subset[a,b]$ such that $\big(\gamma\restrictedto_{[c,d]}\big)^{-1}(S)$ non-empty, i.e., passing through at least one solenoid directly without being deflected. A geometric geodesic can be seen as the uniform limit of a family of regular geodesics in $X$.
\item In particular, the polygonal trajectory $\gamma$ is a \emph{strictly diffractive geodesic} if it is a diffractive geodesic but not a partially geometric geodesic.
\end{itemize}
\end{definition}

\begin{figure}[H]
  \includegraphics[width=0.6\linewidth]{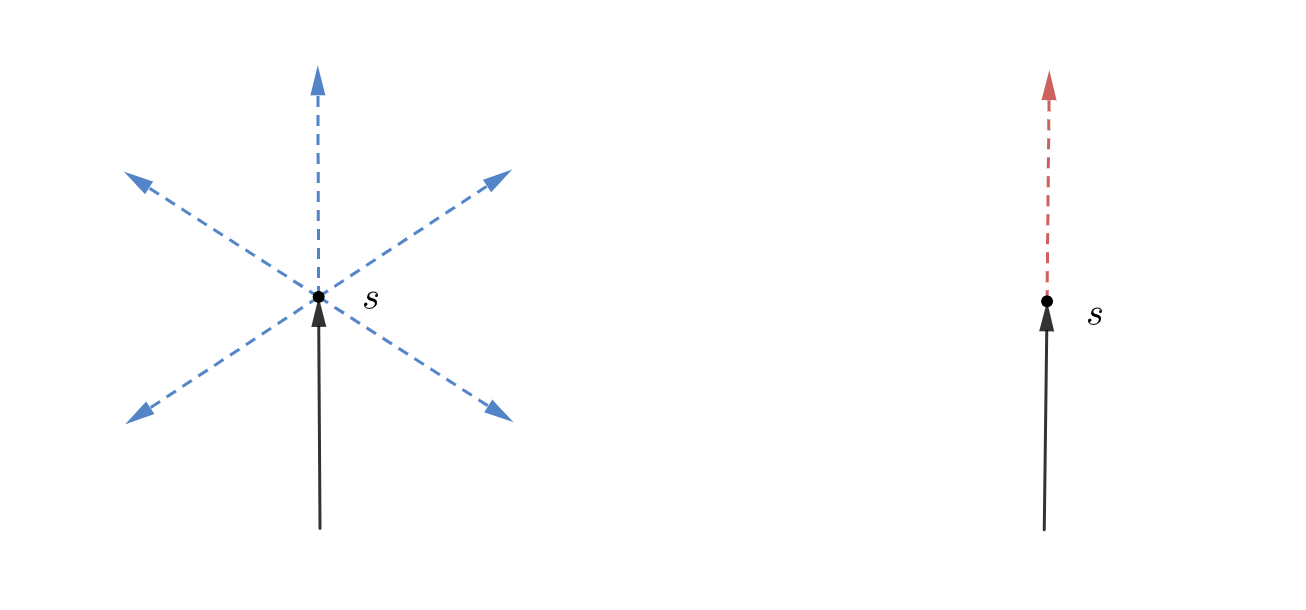}
  \caption{Diffractive and geometric geodesics near $s\in S$}
  \medskip
  \small
   The picture on the left is a family of diffractive geodesics, while the picture on the right is a geometric geodesic passing through $s$.
  \label{Fig1}
\end{figure}

Consider polar coordinates at $s\in S$ and a diffractive geodesic restricted to a neighborhood of $s$. Assume the diffractive geodesic is the concatenation of an incoming ray at the angle $\theta_{\text{in}}$ and an outgoing ray at the angle $\theta_{\text{out}}$. We use $\beta=\theta_{\text{out}}-\theta_{\text{in}}$ to denote the diffraction angle at the solenoid $s$ of such diffractive geodesic. In particular, $\beta=\pm\pi$ correspond to geometric geodesics or non-strictly diffractive geodesics at $s$. 

We shall also use the geodesic flow at the level of the cotangnet bundle $T^*X \cong (\RRS)\times \RR^2$. We also write $S^*X$ for the corresponding cosphere bundle. In the following, we restrict our consideration of the geodesic flow to the interior of $X$ without considering the behavior at the solenoid set $S$; we therefore can employ the standard pseudodifferential calculus on $\RR^2$ rather than b-calculus.

The Hamiltonian vector field of $P_{\aalpha}$ is $\textbf{H}_{\aalpha}=2\xi\partial_x+2\eta\partial_y$, where $\xi,\eta$ are dual variables of $x,y$ correspondingly. Let $\hat{\textbf{H}}_{\aalpha}$ be the Hamiltonian vector field of $P_{\aalpha}$ projected to the cosphere bundle $S^*X$. The integral curves of $\hat{\textbf{H}}_{\aalpha}$ on $S^*X$ are the unit speed polygonal trajectories in $S^*X$ with (possible) jumps in $(\xi,\eta)$-variables at $S$. Furthermore, $\xi,\eta$ are constants on each connected component (a straight-line segment) of an integral curve. In particular, the projection of these (broken) integral curves are the diffractive geodesics we defined before.

Given this background, we may define two symmetric relations between points in $S^*X$: a ``geometric" relation and a ``diffractive'' relation. These correspond to two different possibilities for linking points in $S^*X$ by integral curves of $\hat{\textbf{H}}_{\aalpha}$. Note that since we are on $\RR^2$, there is a canonical metric on $S^*X$.

\begin{definition}
Let $q$ and $q'$ be points in the cosphere bundle $S^*X$.
\begin{enumerate}
    \item We define $q$ and $q'$ to be \emph{diffractively related by time $t$} if there exists an integral curve $\gamma\subset S^*X$ of length $t$ with starting point $q$ and end point $q'$. In particular, for each possible time $t'$ with jump in the $(\xi,\eta)$-variable, the end point at $t'-$ and the starting point at $t'+$ must lie over the same point $s$ of $S$.    
    \item Among points that are diffractively related, we define $q$ and $q'$ to be \emph{geometrically related by time $t$} if there exists a continuous\footnote{There are actually removable discontinuities along $\gamma$ over the set $S$.} integral curve $\gamma\subset S^*X$ of length $t$ with starting point $q$ and end point $q'$.
\end{enumerate}
\end{definition}

Note that the projection of the aforementioned integral curves to $X$ are diffractive geodesics. We can thus relate the integral curves of $\hat{\textbf{H}}_{\aalpha}$ over $X$ to the diffractive geodesics in $X$ as the following:

\begin{proposition}
Suppose that $q$ and $q'$ are points in $S^*X$. Then
\begin{enumerate}
    \item $q$ and $q'$ are diffractively related by time $t$ if and only if they are connected by a lifted diffractive geodesic from $X$ to $S^*X$ of length $t$;
    \item $q$ and $q'$ are geometrically related by time $t$ if and only if they are connected by a lifted geometric geodesic from $X$ to $S^*X$ of length $t$. 
\end{enumerate}
\end{proposition}

\begin{remark}
Due to the above proposition, we sometimes use $\gamma$ to denote either a diffractive geodesic in $X$ or a lifted diffractive geodesic in $S^*X$ when there is no ambiguity. When necessary, we use $\gamma^{\mathfrak{b}}(s)$ to denote the projection of $\gamma(s)\in S^*X$ to the base $X$ for lifted diffractive geodesic $\gamma$, where the upper-right index $\mathfrak{b}$ denotes the projection to the base.
\end{remark}

\subsection{Propagation of singularities}
Before ending this section, we present a result of propagation of diffractive singularities of the Aharonov--Bohm wave propagator in $\RR^2$ with one singular vector potential \cite{yang2020ab}, which essentially states that the diffractive singularities propagate along the diffractive geodesics and are conormal.  
\begin{theorem}
Consider the wave equation 
$$\Box u=0,\ u\rvert_{t=0}=0,\ \partial_t u\rvert_{t=0}=f$$
with the vector potential $\Vec{A}=-\alpha\cdot\left(-\frac{y}{x^2+y^2},\frac{x}{x^2+y^2}\right)^{T}$ where $\alpha\in(0,1)$. Under polar coordinates, define $E(t,r,\theta,r',\theta')$ to be the Schwartz kernel of the sine propagator $\frac{\sin t\sqrt{\pa}}{\sqrt{\pa}}$. Then for $t>r'$, $|\theta-\theta'|\neq\pi$ and near $\{t=r'+r\}$, $E(t,r,\theta,r',\theta')$ is a conormal distribution with respect to $\{t=r'+r\}$.
\end{theorem}

We are particularly interested in closed geodesics when computing the trace ``$\Tr U_{\alpha}(t)$". To avoid the appearance of the geometric geodesic in a closed general geodesic, we assume that there are no three solenoids collinear. Therefore, there is no closed geometric geodesic in $X$; all closed generalized geodesics are closed strictly diffractive geodesics.

\section{Diffractive propagation under a single diffraction}
\label{single_d}
By the finite speed of propagation, the wave emanating from points close to a solenoid, for example $s_1$, can only generate one diffraction for time $t$ small enough. Therefore, in this section, we consider the diffractive wave of the propagator $U_{\aalpha}(t)$ when it only undergoes a single diffraction.

First, we start with $U_{\alpha_1}(t)=e^{it\sqrt{P_{\alpha_1}}}$, which is the propagator corresponding to the Hamiltonian $P_{\alpha_1}$ with one singular vector potential $\Vec{A}_1$ defined by the equation \eqref{potential}. In terms of polar coordinates around the solenoid $s_1$, we have the following proposition regarding the Schwartz kernel $U_{\alpha_1}(t,r_1,\theta_1,r_2,\theta_2)$:
\begin{proposition}
For $\lvert \theta_1 - \theta_2 \rvert \neq \pi$, i.e., away from the geometric wave front, the diffractive part of the wave propagator $U_{\alpha_1}(t)$ of $P_{\alpha_1}$ is a conormal distribution with respect to $\{t=r_1+r_2\}$; locally near $\{t=r_1+r_2\}$, it admits the form: 
\begin{equation*}
     \int_{\RR}e^{i\lambda(t-r_2-r_1)}\tilde{c}(t,r_1,\theta_1,r_2,\theta_2;\lambda)d\lambda
\end{equation*}
whose symbol $\tilde{c}\in S^0(\RR\times X\times X; \RR_{\lambda})$ is 
\begin{equation*}
\rho(\lambda) \cdot \sin(\pi\alpha_1) \cdot \frac{(r_1r_2)^{-\frac{1}{2}}}{2\pi} \cdot \frac{e^{-i\theta_1}+e^{i\theta_2}}{\cos\theta_1+\cos\theta_2}
\end{equation*}
modulo terms in $S^{-\frac{1}{2}+0}$, where $\rho(\lambda)\in \CI(\RR_{\lambda})$ is a smooth function satisfying $\rho\equiv 0$ for $\lambda<0$ and $\rho\equiv 1$ for $\lambda>1$.
\end{proposition}
\begin{proof}
This is a direct corollary of \cite[Theorem 6.1]{yang2020ab} using $e^{it}=\cos t+ i\sin t$ and microlocality of $\sqrt{P_{\aalpha}}$ (cf. Proposition \ref{ap_mlty}). 
\end{proof}
We define
\begin{equation*}
    \tilde{d}_{\alpha_1}(\theta_1,\theta_2)\df\sin(\pi\alpha_1) \cdot \frac{e^{-i\theta_1}+e^{i\theta_2}}{\cos\theta_1+\cos\theta_2}
\end{equation*}
to be the \emph{diffraction coefficient} at the solenoid $s_1$, which is independent of $r_1,r_2$ and $\lambda$.
\begin{remark}
By the identity
$$\frac{e^{-i\theta_1}+e^{i\theta_2}}{\cos\theta_1+\cos\theta_2}=\frac{e^{-i\left(\frac{\theta_2-\theta_1}{2}\right)}}{\cos(\frac{\theta_2-\theta_1}{2})},$$
$\tilde{d}_{\alpha}(\theta_1,\theta_2)$ is invariant for fixed $\beta=\theta_2-\theta_1$ which is the \emph{diffraction angle}. Henceforth, we use the notation $\tilde{d}_{\alpha}(\beta)$ with the diffraction angle $\beta$ instead.
\end{remark}

Now we consider diffraction of the propagator $e^{it\sqrt{P_{\aalpha}}}$ near one solenoid $s_k\in S$ with the singular vector potential $\Vec{A}=\sum_{i=1}^n\Vec{A}_i$. Due to the long range effect of singular vector potentials $\Vec{A}_i$ with $i\neq k$, we need to introduce a gauge transformation to offset this effect without changing the singular vector potential $\Vec{A}_k$ generated by the solenoid where the diffraction happens. 

Consider the Hamiltonian
\begin{equation*}
    P_{\aalpha}=\left(\frac{1}{i}\nabla-\Vec{A}\right)^2
\end{equation*}
with $\Vec{A}=\sum_{i=1}^{n}\Vec{A}_i$ where $\Vec{A_i}$ given by equation \eqref{potential}. Taking complex coordinates in $\RR^2$ with $z=x+iy$, we introduce an angular function $\phi_i$ with respect to the solenoid $s_i$ as 
\begin{equation}
\label{ang_i}
    \phi_i(x,y)=\arg (z-s_i).
\end{equation}

\begin{remark}
\label{angle}
We should note that although $\phi_i(z)$ is defined as a sheaf or, in other words, defined on the logarithmic covering of $\RR^2\backslash\{s_i\}$, it still can be used invariantly since both its gradient $\nabla\phi_i(x,y)$ and the diffraction angle $\phi_i(x_2,y_2)-\phi_i(x_1,y_1)$ are invariantly well-defined. The diffraction coefficient and the choice of gauge really depends on the gradient or the difference rather than $\phi_i$ itself.  
\end{remark}

Note $\Vec{A_i}$ is therefore given by the gradient of the angle function $\phi_i(x,y)$ and the magnetic flux $\alpha_i$: 
\begin{equation}
    \Vec{A}_i=-\alpha_i \nabla \phi_i. 
\end{equation}
Away from the solenoid $s_j\in S$, the angular function $\phi_j$ is smooth and $\Vec{A}_j$ is curl-free. Thus the gauge transformations by adding a vector potential $\Vec{A}_j$ near $s_k$ for $j\neq k$ does not change the magnetic field at $s_k$.
In particular, near the $k$-th solenoid $s_k$ the local gauge transformations yield the relation:
\begin{equation}
    \label{gauge}
    \left(\frac{1}{i}\nabla-\Vec{A}\right)^2=e^{-i\left(\sum_{j\neq k}\alpha_j\phi_j\right)}\left(\frac{1}{i}\nabla-\Vec{A_k}\right)^2e^{i\left(\sum_{j\neq k}\alpha_j\phi_j\right)}.
\end{equation}
Note that the propagator $U_{\aalpha}(t)$ is given by the solution operator to the half-wave equation:
\begin{equation}
    \begin{split}
        (D_t&-\sqrt{P_{\aalpha}})u=0\\
            &u\rvert_{t=0}=f,
    \end{split}
\end{equation}
and the microlocality\footnote{See the Appendix for a brief discussion on microlocality of $\sqrt{P_{\aalpha}}$} of $\sqrt{P_{\aalpha}}$ allows us to consider the propagator $U_{\aalpha}(t)$ (micro)locally. Therefore, locally near the diffraction at $s_k$, the propagator $U_{\aalpha}(t)$ is related to $U_{\alpha_k}(t)$ by conjugation by a phase-shift term as in \eqref{gauge}:
$$U_{\aalpha}(t)=e^{-i\left(\sum_{j\neq k}\alpha_j\phi_j\right)} U_{\alpha_k}(t) e^{i\left(\sum_{j\neq k}\alpha_j\phi_j\right)}.$$
Summarizing the result in terms of the Schwartz kernel, we obtain the following proposition:
\begin{proposition}
\label{prop_pgtr1}
For $\lvert \theta_1 - \theta_2 \rvert \neq \pi$, the diffractive part of the propagator $U_{\aalpha}(t)$ of $P_{\aalpha}$ near the $k$-th solenoid $s_k$ is a conormal distribution with respect to $\{t=r_1+r_2\}$; locally near  $\{t=r_1+r_2\}$, it admits a Fourier integral expression:
\begin{equation}
\label{pgtr1}
     \int_{\RR}e^{i\lambda(t-r_2-r_1)}c(t,r_1,\theta_1,r_2,\theta_2;\lambda)d\lambda
\end{equation}
where the symbol $c\in S^0(\RR\times X\times X; \RR_{\lambda})$ is 
\begin{equation}
\label{symbol_pgtr1}
   \frac{(r_1r_2)^{-\frac{1}{2}}}{2\pi}\cdot\tilde{d}_{\alpha_k}(\beta)\cdot e^{-i\left(\sum_{j\neq k}\alpha_j\cdot\left(\phi_j(r_2,\theta_2)-\phi_j(r_1,\theta_1)\right)\right)}\cdot \rho(\lambda)
\end{equation}
modulo elements in $S^{-\frac{1}{2}+0}$, where $\rho(\lambda)\in \CI(\RR_{\lambda})$ is a smooth function satisfying $\rho\equiv 0$ for $\lambda<0$ and $\rho\equiv 1$ for $\lambda>1$.
\end{proposition}

\begin{remark}
As we mentioned in Remark \ref{angle}, the angle difference $\phi_j(r_2,\theta_2)-\phi_j(r_1,\theta_1)$ is well-defined and independent of the choice of $\phi$. 
\end{remark}

\begin{remark}
\label{rmk-pos}
Combining this proposition with the standard propagation of singularities in the smooth setting, we obtain that the singularities propagate along all diffractive geodesics with unity speed.
\end{remark}

Taking both the phase-shift, which is due to the global topological effect, and the diffraction coefficient into account, we define
\begin{equation}
\label{d-coeff}
    d_{\alpha_k}(\beta,z_1,z_2)\df\tilde{d}_{\alpha_k}(\beta)\cdot e^{-i\left(\sum_{j\neq k}\alpha_j\cdot\left(\phi_j(r_2,\theta_2)-\phi_j(r_1,\theta_1)\right)\right)},
\end{equation}
which is independent of $\lambda$. Note that we sometimes omit $z_1,z_2$ and write $d_{\alpha_k}(\beta)$ when there is no ambiguity. We also introduce the following definition:
\begin{definition}
A \emph{diffractive geodesic triple} $(\gamma, \Pi_s, \Pi_e)$ is a lifted diffractive geodesic $\gamma\subset S^*X$ with length $t$ and a pair of pseudodifferential operators $\Pi_s$ and $\Pi_e$ of order zero such that $\gamma(0)\in\WF'\Pi_s$ and $\gamma(t)\in\WF'\Pi_e$.\footnote{$s,\ e$ stand for starting and ending. Also note that $\gamma(0),\gamma(t)$ are always away from the solenoid since $X=\RR^2\backslash S$.}
\end{definition}

Using such a diffractive geodesic triple, we can microlocalize the propagator $U_{\aalpha}(t)$ along the diffractive geodesic $\gamma$. Assume $\gamma$ undergoes one diffraction at $s_k$. Consider a microlocalized propagator $U_{\aalpha, \gamma}(t)\df \Pi_e U_{\aalpha}(t) \Pi_s$ along $\gamma$ associated with the diffractive geodesic triple $(\gamma, \Pi_s, \Pi_e)$. For the diffraction angle $\beta$ different from $\pm\pi$, we can always shrink the size of $\WF'\Pi_s, \WF'\Pi_e$ such that none of their points are geometrically related by time $t$ if necessary. Indeed, in what follows we shall always assume that this is the case if the starting point and end point are not geometrically related. The pseudodifferential operators $\Pi_s, \Pi_e$ are given by 
$$\Pi_s=\frac{1}{(2\pi)^2}\int e^{i(\tilde{z}'-z')\zeta'}a_s(\tilde{z}',\zeta')d\zeta',$$
$$\Pi_e=\frac{1}{(2\pi)^2}\int e^{i(z-\tilde{z})\zeta}a_e(z,\zeta)d\zeta$$
with $a_s,a_e\in S^0_c(T^*X)$. In polar coordinates centered at $s_k$, $z=(r_z,\theta_z)$ and $z'=(r_{z'},\theta_{z'})$. 

\begin{lemma}
\label{lemma_1_diffraction}
For $\lvert \theta_z - \theta_{z'} \rvert \neq \pi$, the microlocalized propagator $U_{\aalpha, \gamma}(t)\df \Pi_e U_{\aalpha}(t) \Pi_s$ associated with the diffractive geodesic triple $(\gamma, \Pi_s, \Pi_e)$ is an \emph{FIO} with the wavefront relation $\WF' \Pi_e U_{\aalpha}(t) \Pi_s$ given by points in $\WF'\Pi_s\times\WF'\Pi_e$ that are diffractively related by time $t$. Under polar coordinates near $s_k$,
\begin{equation}
    \label{canonical}
    \begin{split}
    \WF'{\Pi_e U_{\aalpha}(t) \Pi_s}=\Big\{ &(r_z=t-r_{z'},\theta_z,r_{z'}, \theta_{z'}; -\lambda, 0, \lambda, 0)\in T^*(X\times X) \\ &\ \ \ \restrictedto \  (r_{z'},\theta_{z'};\lambda,0)\in \WF'\Pi_s, (r_{z},\theta_{z};-\lambda,0)\in \WF'\Pi_e \Big\}.
\end{split}
\end{equation}
Moreover, $U_{\aalpha, \gamma}(t)$ locally admits an oscillatory integral form:
\begin{equation}
    \label{FIO_MLpgtr}
    \int e^{i\lambda(t-r_z-r_{z'})}c_{s,e}(t,r_z,r_{z'},\theta_z,\theta_{z'}; \lambda) d\lambda
\end{equation}
where the symbol $c_{s,e}\in S^0(\RR\times X\times X; \RR_{\lambda})$ is given by 
\begin{equation}
\label{sym_MLpgtr1}
     \frac{(r_zr_{z'})^{-\frac{1}{2}}}{2\pi} \cdot d_{\alpha_k}(\theta_z-\theta_{z'})\cdot a_s(z',\lambda\theta_{z'})\cdot a_e(z,-\lambda\theta_z)\cdot \rho(\lambda)
\end{equation}
modulo elements of $S^{-\frac{1}{2}+0}$, with $\rho(\lambda)\in \CI(\RR_{\lambda})$ being a smooth function satisfying $\rho\equiv 0$ for $\lambda<0$ and $\rho\equiv 1$ for $\lambda>1$.
\end{lemma}

\begin{proof}
Applying the theory of \text{FIOs}\cite{hormander1971fourier} \cite{duistermaat1972fourier} to the propagator \eqref{pgtr1}, the canonical relation of the diffractive part of $U_{\aalpha}(t,z,z')$ near $s_k$ is given by
\begin{equation*}
    \text{C}=\Big\{ (r_z=t-r_{z'},\theta_z,r_{z'}, \theta_{z'}; -\lambda, 0, \lambda, 0)\in T^*(X\times X)\restrictedto \ \lambda\neq0\Big\}. 
\end{equation*}
In view of the composition of canonical relations, the wavefront relation of the Schwartz kernel of $\Pi_e U_{\aalpha}(t) \Pi_s$ is given by
\begin{equation*}
\begin{split}
    \WF' {\Pi_e U_{\aalpha}(t) \Pi_s}=\Big\{ &(r_z=t-r_{z'},\theta_z,r_{z'}, \theta_{z'}; -\lambda, 0, \lambda, 0)\in T^*(X\times X) \\ &\ \ \ \restrictedto \  (r_{z'},\theta_{z'};\lambda,0)\in \WF'\Pi_s, (r_{z},\theta_{z};-\lambda,0)\in \WF'\Pi_e \Big\}. 
\end{split}
\end{equation*}
Now we apply the stationary phase lemma to compute the local expression of the composition. Composing the Schwartz kernels yields
\begin{equation*}
    \Pi_e U_{\aalpha}(t) \Pi_s=\frac{1}{(2\pi)^4}\int e^{i\phi}a_e(z,\zeta)c(t,\tilde{z}',\tilde{z}; \lambda)a_s(\tilde{z}',\zeta')d\lambda d\tilde{z}'d\tilde{z} d\zeta'd\zeta
\end{equation*}
where $\phi=(\tilde{z}'-z')\cdot\zeta'+\lambda(t-r_{\tilde{z}'}-r_{\tilde{z}})+(z-\tilde{z})\cdot\zeta$. Since the critical points of the phase function are
\begin{equation*}
    \tilde{z}'=z',\ \tilde{z}=z,\ \zeta'=\lambda \theta_{z'},\ \zeta=-\lambda\theta_{z},
\end{equation*}
applying method of stationary phase in $(\tilde{z}',\tilde{z}, \zeta', \zeta)$-variables yields
\begin{equation}
    \Pi_e U_{\aalpha}(t) \Pi_s= \int e^{i\lambda(t-r_z-r_{z'})}c_{s,e}(t,r_z,r_{z'},\theta_z,\theta_{z'}; \lambda) d\lambda
\end{equation}
where the symbol $c_{s,e}\in S^0(\RR\times X\times X; \RR_{\lambda})$ is given by 
\begin{equation}
    a_e(z;-\lambda\theta_z) \cdot c(t, z',z; \lambda) \cdot a_s(z';\lambda\theta_{z'})\cdot \rho(\lambda)
\end{equation}
modulo elements in $S^{-\frac{1}{2}+0}$.
\end{proof}

\section{Microlocalized propagation under multiple diffractions}
\label{sec_mlpg}
In this section, we construct a microlocalized propagator $\Pi_e U_{\aalpha,\gamma}(t) \Pi_s$ associated with a diffractive geodesic triple $(\gamma, \Pi_s, \Pi_e)$ under multiple diffractions. 

Recall by definition, for the diffractive geodesic $\gamma\subset T^*X$ with length $t$, $\Pi_s$ and $\Pi_e$ are $0$-th order pseudodifferential operators such that $\gamma(0)\in \WF' \Pi_s$ and $\gamma(t)\in \WF' \Pi_e$. In addition, we can take $\Pi_s,\Pi_e\in\Psi^0_c(X)$ such that singular supports $\WF' \Pi_s$ and $\WF' \Pi_e$ can be as small as we want. To construct the microlocalized propagator, we want to insert a microlocal cutoff between each diffraction along the diffractive geodesic $\gamma$. Moreover, we need the principal symbols of these interim pseudodifferential operators to be identically $1$ near the forward/backward flow-out of $\WF'\Pi_s$/$\WF'\Pi_e$ for our purposes. 

Assume there are $n\in\mathbb{N}$ diffractions along $\gamma$ at times $0<S_1<S_2<\cdots<S_n<t$. Take $T_i\in (S_i,S_{i+1})$ for $1\leq i\leq n-1$ and $T_0=S_0=0$. Define $t_i\df T_i-T_{i-1}$ to be the propagation time between $(T_{i-1},T_i)$. We first construct $\Pi_1$ near $\gamma(T_1)$; the construction of the rest intermediate microlocalizers $\Pi_i$ can be carried out inductively.

Consider two sufficiently small neighborhoods $U_1, V_1$ such that $\pi_L \big(\WF' {U_{\aalpha}(T_1)\circ\Pi_s}\big)\cap \gamma\subset U_1\subset V_1\subset T^*X$, where $\pi_L: T^*(X\times X)\rightarrow T^*(X)$ is the projection to the first copy of $T^*X\subset T^*(X\times X)$. We can therefore choose a $0$-th order pseudodifferential operator $\Pi_1$ such that its principal symbol $\sigma_0(\Pi_1)\restrictedto_{U_1}\equiv 1$ and $\sigma_0(\Pi_1)\restrictedto_{V_1^{c}}\equiv 0$. Assuming $\Pi_{k-1}$ has been constructed, $\Pi_k$ can be constructed similarly such that $\sigma_0(\Pi_k)\restrictedto_{U_k}\equiv 1$ and $\sigma_0(\Pi_k)\restrictedto_{V_k^c}\equiv 0$, where $U_k, V_k$ are neighborhoods of $\pi_L \big(\WF' {U_{\aalpha}(T_k)\circ\Pi_{k-1}}\big)\cap \gamma$. The \emph{microlocalized propagator} associated with a diffractive geodesic triple $(\gamma,\Pi_s,\Pi_e)$ is thus defined as 
\begin{equation}
\label{ML-pgtr}
\Pi_e U_{\aalpha,\gamma}(t) \Pi_s \df \Pi_e U_{\aalpha}(t-T_{n-1})\Pi_{n-1}\cdots\Pi_1 U_{\aalpha}(t_1) \Pi_s.
\end{equation}

\begin{remark}
As we shall see later, the definition of microlocalized propagators is independent of the interim microlocalizers modulo smoothing operators. In particular, it only depends on the initial and the final microlocal cutoffs. 
\end{remark}

\begin{figure}[H]
  \includegraphics[width=0.95\linewidth]{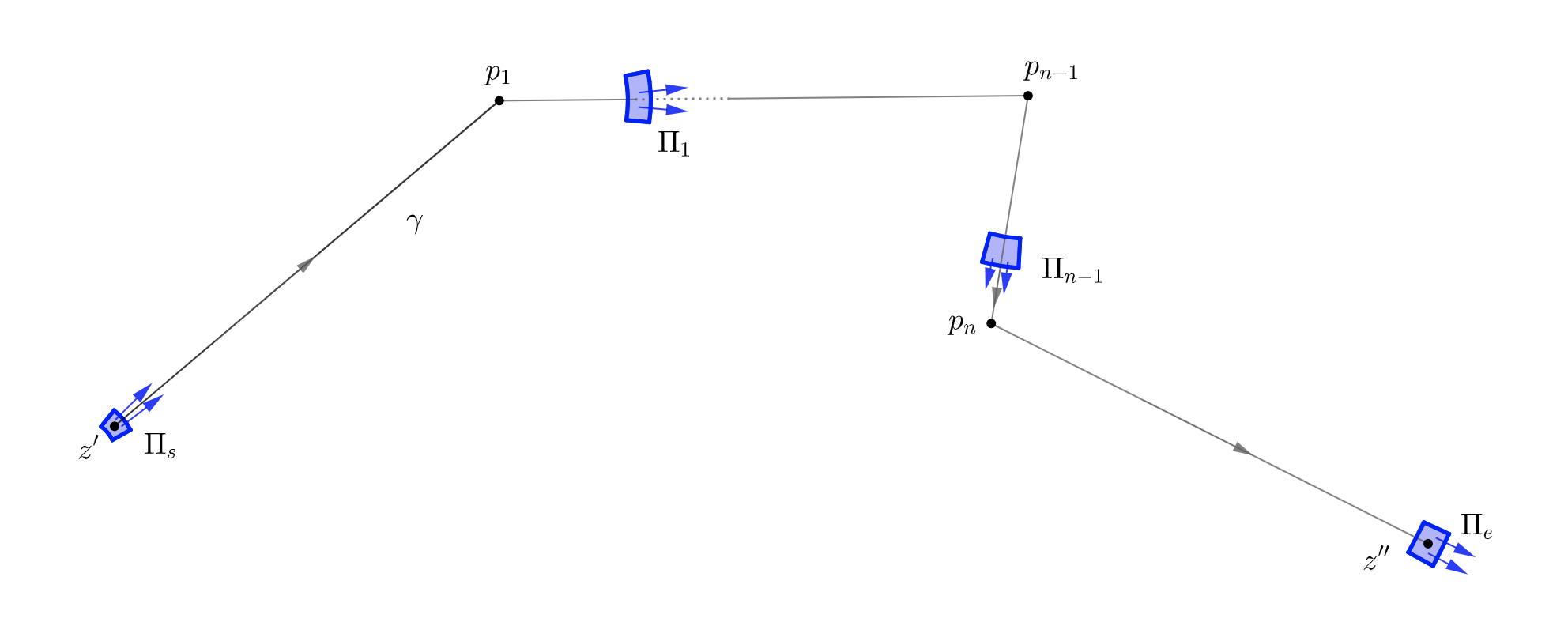}
  \caption{Microlocalized propagator $\Pi_e U_{\aalpha,\gamma}(t) \Pi_s$}
  \medskip
  \small
  \label{Fig2}
\end{figure}

Now we compute the microlocalized propagator associated with $(\gamma,\Pi_s,\Pi_e)$ using composition of FIOs. We first introduce some notations and conventions. Assume $\gamma$ has $n$ diffraction points $p_1,\cdots, p_n$ and total length $L$. We take polar coordinates around $p_1$ with the starting point $z'=(r_{z'}, \theta_{z'})$ and polar coordinates around $p_n$ with the end point $z''=(\bar{r}_{z''}, \bar{\theta}_{z''})$. Assume the $i$-th diffraction at $p_i$ has diffraction angle $\beta_i$ and the distance between $p_i$ and $p_{i+1}$ is $l_i$. We also define the following notations:
\begin{equation*}
    d_{\gamma,\aalpha}\df \prod_{i=1}^{n}d_{\alpha_i}(\beta_i),\  l_{\gamma,\aalpha} \df r_{z'}\left(\prod_{i=i}^{n-1}l_i\right)\bar{r}_{z''}.
\end{equation*}

\begin{proposition}
\label{prop_pgtr2}
The Schwartz kernel of the microlocalized propagator $\Pi_e U_{\aalpha,\gamma}(t) \Pi_s$ associated with $(\gamma,\Pi_s,\Pi_e)$, which is defined in \eqref{ML-pgtr}, is a conormal distribution with respect to $\{t=r_{z'}+\sum_{i=1}^{n-1}l_i+\bar{r}_{z''}\}$. Locally near $\{t=r_{z'}+\sum_{i=1}^{n-1}l_i+\bar{r}_{z''}\}$, it admits an oscillatory integral representation:
\begin{equation}
\label{int_pgtr2}
    U_{\aalpha,\gamma}(t,z'',z')= \int_{\RR}e^{i\lambda\left(t-r_{z'}-\sum_{i=1}^{n-1}l_i-\bar{r}_{z''}\right)}b(t,\bar{r}_{z''},r_{z'},\bar{\theta}_{z''},\theta_{z'};\lambda)d\lambda
\end{equation}
with the symbol $b\in S_c^{-\frac{n-1}{2}}(\RR\times X\times X; \RR_{\lambda})$ given by 
\begin{equation}
\label{symb_pgtr2}
    (2\pi)^{\frac{n-3}{2}} \cdot e^{-i\frac{(n-1)\pi}{4}} \cdot \frac{d_{\gamma,\aalpha}}{l_{\gamma,\aalpha}^{\frac{1}{2}}} \cdot a_s(z';\lambda\theta_{z'}) \cdot a_e(z'';-\lambda\bar{\theta}_{z''}) \cdot \rho(\lambda) \cdot \lambda^{-\frac{n-1}{2}}
\end{equation}
modulo elements of $S^{-\frac{n}{2}+0}$, where $\rho(\lambda)\in \CI(\RR_{\lambda})$ is a smooth function satisfying $\rho\equiv 0$ for $\lambda<0$ and $\rho\equiv 1$ for $\lambda>1$.
\end{proposition}

\begin{proof}
We compute the microlocalized propagator by applying Lemma \ref{lemma_1_diffraction} inductively. We start with the situation of two diffractions; this involves composing two single diffraction propagators in Lemma \ref{lemma_1_diffraction} once. The microlocalized propagator is 
$$\Pi_e U_{\aalpha}(t-t_1) \Pi_1 U_{\aalpha}(t_1) \Pi_s,$$
where $\Pi_1$ has principal symbol $\sigma_0(\Pi_1)=a_1$. Without loss of generality, we assume that there exists a pseudodifferential operator $\sqrt{\Pi_1}\in\Psi^{0}_c(X)$ with principal symbol $\tilde{a}_1$ such that $\Pi_1-(\sqrt{\Pi_1})^2\in\Psi^{-\infty}(X)$.

Use $(r,\theta)$ (resp. $(\bar{r},\bar{\theta})$) to denote the polar coordinates centered at the first (resp. the second) diffraction. In particular, $\theta=0$ and $\bar{\theta}=0$ correspond to the points on the line segment connecting two diffractions. By Lemma \ref{lemma_1_diffraction}, we obtain

$$\sqrt{\Pi_1} U_{\aalpha}(t_1) \Pi_s= \int e^{i\lambda\big(t_1-r_z-r_{z'}\big)}c_{s,1}(t_1,r_z,r_{z'},\theta_z,\theta_{z'}; \lambda) d\lambda $$
where
\begin{equation*}
    \begin{split}
        c_{s,1}(t_1,r_z,&r_{z'},\theta_z,\theta_{z'}; \lambda)\\
        &\equiv  \frac{(r_zr_{z'})^{-\frac{1}{2}}}{2\pi} \cdot d_{\alpha_1}(\theta_z-\theta_{z'})\cdot a_s(z',\lambda\theta_{z'})\cdot \tilde{a}_1(z,-\lambda\theta_z)\cdot \rho(\lambda) \mod S^{-\frac{1}{2}+0},
    \end{split}
\end{equation*}
and
\begin{equation*}
    \Pi_e U_{\aalpha}(t-t_1) \sqrt{\Pi_1}= \int e^{i\mu\big((t-t_1)-\bar{r}_z-\bar{r}_{z''}\big)}c_{1,e}(t-t_1,\bar{r}_z,\bar{r}_{z''},\bar{\theta}_z,\bar{\theta}_{z''}; \mu) d\mu
\end{equation*}
where
\begin{equation*}
    \begin{split}
        c_{1,e}(t-t_1,\bar{r}_z,&\bar{r}_{z''},\bar{\theta}_z,\bar{\theta}_{z''}; \mu)\\
        &\equiv  \frac{(\bar{r}_z\bar{r}_{z''})^{-\frac{1}{2}}}{2\pi} \cdot d_{\alpha_1}(\bar{\theta}_z-\bar{\theta}_{z''})\cdot \tilde{a}_1(z,\mu\bar{\theta}_{z})\cdot a_e(z'',-\mu\bar{\theta}_{z''})\cdot \rho(\mu) \mod S^{-\frac{1}{2}+0}.
    \end{split}
\end{equation*}
Consider the composition of the above two FIOs. 
\begin{equation*}
\begin{split}
        \Pi_e U_{\aalpha}(t-t_1) \Pi_1 U_{\aalpha}(t_1) \Pi_s
       & \equiv \Pi_e U_{\aalpha}(t-t_1) \sqrt{\Pi_1}\circ \sqrt{\Pi_1} U_{\aalpha}(t_1) \Pi_s \mod \dom_{\aalpha}^{-\infty} \longrightarrow \dom_{\aalpha}^{\infty}\\
       &=\int e^{i\phi}
       c_{1,e}(t-t_1,\bar{r}_z,\bar{r}_{z''},\bar{\theta}_z,\bar{\theta}_{z''}; \mu) c_{s,1}(t_1,r_z,r_{z'},\theta_z,\theta_{z'}; \lambda)dz d\mu d\lambda,
\end{split}
\end{equation*}
with the phase function $\phi=\lambda\big(t_1-r_z-r_{z'}\big)+\mu\big((t-t_1)-\bar{r}_z-\bar{r}_{z''}\big).$ 
Now we apply stationary phase lemma in $(z,\mu)$-variables, or equivalently, in $(r_z,\theta_z, \mu)$-variables. Using the law of cosines, we can write $\bar{z}_r$ in terms of $r_z$ and $\theta_z$. Computing critical points of $\phi$ yields:
\begin{equation}
    \bar{r}_z=(t-t_1)-\bar{r}_{z''},\ \theta_z=0,\ \mu=\lambda.
\end{equation}
This shows there exists a unique critical point lies exactly on the geodesic connecting $p_1$ and $p_2$. Moreover, a simple computation gives
\begin{equation}
\phi''_{\mu,r_z,\theta_z}=
\begin{pmatrix}
0 & 1 & 0\\
1 & 0 & 0\\
0 & 0 & -\frac{\lambda\cdot l\cdot r_z}{\bar{r}_z}
\end{pmatrix}    
\end{equation}
and $\sgn(\phi''_{\mu,r_z,\theta_z})=-1$. The stationary phase lemma thus yields
\begin{equation}
    \Pi_e U_{\aalpha}(t-t_1) \Pi_1 U_{\aalpha}(t_1) \Pi_s= \int e^{i\lambda(t-l-r_{z'}-\bar{r}_{z''})}b_{s,1,e}(t,\bar{r}_{z''},r_{z'},\bar{\theta}_{z''},\theta_{z'};\lambda)d\lambda
\end{equation}
where 
\begin{equation}
\label{principal_b_01e}
\begin{split}
    b_{s,1,e}\equiv \  &(2\pi)^{-\frac{1}{2}} \cdot e^{-i\frac{\pi}{4}}\cdot \frac{d_{\alpha_1}(\beta_1) \cdot d_{\alpha_2}(\beta_2)}{(r_{z'} \cdot l \cdot \bar{r}_{z''} )^{\frac{1}{2}}} \cdot a_s(z';\lambda\theta_{z'}) \\ 
    &\times a_1(z;-\lambda\theta_z) \cdot a_e(z'';-\lambda\bar{\theta}_{z''}) \cdot \rho^2(\lambda) \cdot \lambda^{-\frac{1}{2}}\restrictedto_{\{\bar{r}_z=(t-t_1)-\bar{r}_{z''},\ \theta_z=0\}} \mod S^{0}.
\end{split}
\end{equation}
By the construction of $\Pi_1$, near the singularities at $\{t=r_{z'}+l+\bar{r}_{z''}\}$, the principal symbol $\sigma_0(a_1)\equiv 1$. Therefore, modulo lower order terms, $a_1$ can be omitted in the expression \eqref{principal_b_01e}. This finishes the first step of the induction. 

Assuming we have showed \eqref{int_pgtr2} and \eqref{symb_pgtr2} for the microlocalized propagator with $k$ diffractions 
$$\Pi_k U_{\aalpha}(t-T_{k-1})\Pi_{k-1}\cdots\Pi_1 U_{\aalpha}(t_1) \Pi_s.$$
The microlocalized propagator with $(k+1)$ diffractions
$$\Pi_{k+1} U_{\aalpha}(t-T_{k})\Pi_{k}\cdots\Pi_1 U_{\aalpha}(t_1) \Pi_s$$
is the composition of $\Pi_{k+1} U_{\aalpha}(t-T_{k})$ and $\Pi_k U_{\aalpha}(t_{k})\Pi_{k-1}\cdots\Pi_1 U_{\aalpha}(t_1) \Pi_s$. Using the composition of FIOs and the method of stationary phase in the intermediate $(\mu,r_z,\theta_z)$-variables as before, it is straightforward to show the formula holds for $(k+1)$ diffractions. 
\end{proof}

\begin{remark}
This proposition also shows that each diffraction by a singular vector potential improves the regularity by order $\frac{1}{2}$ as the diffraction on manifolds with conic singularities. See \cite{hillairet2005contribution} and \cite{ford2017diffractive} for detailed discussion in those settings.   
\end{remark}

In the next lemma, we show that by microlocalizing near the starting point and the end point, the wave propagator $U_{\aalpha}(t)$ is equivalent to microlocalized propagators $U_{\aalpha,\gamma}(t)$. 

\begin{lemma}
\label{lemma_pgtr_equiv}
For any $L>0$, choosing any two pseuododifferential operators $\Pi_s,\Pi_e\in \Psi^0_c(X)$ with $\WF'\Pi_s$, $\WF'\Pi_e$ microlocally small enough, near $t=L$, 
\begin{enumerate}
\item $\Pi_e U_{\aalpha}(t) \Pi_{s}: \dom_{\aalpha}^{-\infty} \longrightarrow \dom_{\aalpha}^{\infty}$, i.e., is a smoothing operator, if there is no diffractive geodesic $\gamma$ such that $\gamma(0)\in \WF'\Pi_s$ and $\gamma(t)\in \WF'\Pi_e$. 
\item Otherwise, if there exists $N$ diffractive geodesics $\{\gamma_i\}_{i=1}^N$ such that $\gamma_i(0)\in \WF'\Pi_s$ and $\gamma_i(t)\in \WF'\Pi_e$ for $1\leq i\leq N$, then the propagator can be decomposed to the sum of $N$ microlocalized propagators $\{U_{\aalpha,\gamma_i}(t)\}_{i=1}^N$ with each $U_{\aalpha,\gamma_i}(t)$ associated with $(\gamma_i, \Pi_s,\Pi_e)$ such that
\begin{equation}
    \label{pgtr_equiv}
    \Pi_e U_{\aalpha}(t) \Pi_{s} \equiv  \sum_{i=1}^N\Pi_e U_{\aalpha,\gamma_i}(t) \Pi_{s} \mod \dom_{\aalpha}^{-\infty} \longrightarrow \dom_{\aalpha}^{\infty}.
\end{equation}
\end{enumerate}
\end{lemma}

\begin{remark}
\label{rmk_solenoid_cutoff}
Part (1) of the lemma also holds if we replace $\Pi_s$ and $\Pi_e$ by smooth cutoff functions $\psi,\varphi \in \CI_c(X)$ respectively such that no diffractive geodesic $\gamma$ with $\gamma(0)\in \supp \psi$ and $\gamma(t)\in \supp \varphi$. The proof again follows from the propagation of singularities as in the case of microlocal cutoffs.
\end{remark}

\begin{proof}[Proof of Lemma \ref{lemma_pgtr_equiv}]
The proof is essentially repetitively using the fact that the composition of the canonical relations is the canonical relations of the composition of two FIOs \cite{hormander1971fourier} \cite{duistermaat1972fourier}.

Part (1) follows directly from Remark \ref{rmk-pos}. For part (2), we prove it by induction on the number of diffractions. We show relation \eqref{pgtr_equiv} for all triples $\{(\gamma_i, \Pi_s,\Pi_e)\}$. Note that if $\{\gamma_i\}=\emptyset$, then (2) implies (1). If $\Pi_e U_{\aalpha}(t) \Pi_{s}$ involves at most one diffraction along all possible diffractive geodesics, the microlocal smallness of $\WF'\Pi_s$ and $\WF'\Pi_e$ can therefore restrict the diffraction to one solenoid. Lemma \ref{lemma_1_diffraction} therefore gives the canonical relations and the local representation of the FIO. 

Assume $\Pi_e U_{\aalpha}(t) \Pi_{s}$ involves at most two diffractions along all the possible diffractive geodesics. As before, there is also at most one continuous family (with starting/ending point varying continuously in $\WF'\Pi_s$/$\WF'\Pi_e$ and the part between the first and last diffraction fixed) of diffractive geodesics by the microlocal smallness of $\WF'\Pi_s$ and $\WF'\Pi_e$. Choosing one diffractive geodesic $\gamma$ among this family, we can construct the microlocalized propagator 
$$\Pi_e U_{\aalpha,\gamma}(t) \Pi_{s} = \Pi_e U_{\aalpha}(t_2) \Pi_1 U_{\aalpha}(t_1) \Pi_{s}$$
with $t=t_1+t_2$ as in the beginning of this section. In order to show equation \eqref{pgtr_equiv}, it suffices to show 
$$\Pi_e U_{\aalpha}(t_2) (\Id-\Pi_1) U_{\aalpha}(t_1) \Pi_{s}$$
is a smoothing operator. Using the same techniques as the construction of the elliptic parametrix (c.f. \cite[Section 5.5]{shubin1987pseudodifferential}), we construct $A\in\Psi^0(X)$ such that $(\Id-\Pi_1)-A^2\in\Psi^{-\infty}(X)$. Therefore, 
\begin{equation}
   \label{comp_canonical}
   \Pi_e U_{\aalpha}(t_2) (\Id-\Pi_1) U_{\aalpha}(t_1) \Pi_{s} \equiv \Pi_e U_{\aalpha}(t_2) A \circ A U_{\aalpha}(t_1) \Pi_{s}
\end{equation}
modulo smoothing operators. By Lemma \ref{lemma_1_diffraction}, the canonical relation of the FIO $\Pi_e U_{\aalpha}(t_2) A$ (resp. $A U_{\aalpha}(t_1) \Pi_s$) is given by the product of diffractively related points of time $t_2$ (resp. $t_1$) in $\WF'\Pi_e\times \WF'A$ (resp. $\WF'A\times \WF'\Pi_s$). By the construction of the microlocalized propagator, the principal symbol $\sigma_0(\Pi_1)\equiv 1$ near the intersection of $t_1$-flow-out of $\WF'\Pi_s$ with $\gamma$. Therefore, there is no point in $\WF'A\cap\gamma$ such that it is diffractive related to $\WF'\Pi_s$ by time $t_1$. Since all points that are both diffractively related to $\WF'\Pi_s$ and $\WF'\Pi_e$ are in $\gamma$, this yields
\begin{equation*}
    \WF' {\Pi_e U_{\aalpha}(t_2) A \circ A U_{\aalpha}(t_1) \Pi_{s}}
    =\WF' {\Pi_e U_{\aalpha}(t_2) A} \circ
    \WF' {A U_{\aalpha}(t_1) \Pi_{s}}
    =\emptyset.
\end{equation*}
Therefore, the wavefront relation of the LHS of \eqref{comp_canonical} is trivial and the corresponding operator is smoothing. 

Assume we have showed \eqref{pgtr_equiv} for (at most) $k$ diffractions. For $(k+1)$ diffractions. Without loss of generality, we assume there are finitely many continuous families of diffractive geodesics; each family corresponds to a diffractive geodesic $\{\gamma_i\}$ with $\gamma_i(0)\in \WF'\Pi_s$ and $\gamma_i(t)\in \WF'\Pi_e$. In particular, we can shrink $\WF'\Pi_s$ and $\WF'\Pi_e$ such that all $\gamma_i$ have the same segments before the first and after the last diffractions. Choose a time $\hat{t}$ before the last diffraction but after all the penultimate diffractions, then construct the microlocalized propagator for each diffractive geodesic in $\{\gamma_i\}$ with $\{\hat{\Pi}_l\}$ being the microlocalizers at the corresponding point $\gamma_i(\hat{t})$. (There might be subsets $\{\gamma_{i_k}\}\subset \{\gamma_{i}\}$ such that $\gamma_{i_k}(\hat{t})$ corresponds to the same point for each subset, so $\hat{\Pi}_l$ is the microlocalizer corresponding to all $\gamma_i$ in such subset.) Then
\begin{equation}
    \begin{split}
        \Pi_e U_{\aalpha}(t) \Pi_{s} 
        &\equiv  \  \Pi_e U_{\aalpha}(t-\hat{t}) \left(\sum_l \hat{\Pi}_l\right) U_{\aalpha}(\hat{t}) \Pi_{s}\\
        &\equiv \  \sum_l \Pi_e U_{\aalpha}(t-\hat{t}) \hat{\Pi}_l \left(\sum_{k_l} U_{\aalpha, \gamma_{l,k_l}}(\hat{t}) \right) \Pi_{s}\\
        &= \  \sum_{i} \Pi_e  U_{\aalpha, \gamma_{i}}(t) \Pi_{s}
    \end{split}
\end{equation}
modulo smoothing terms. The first equality follows from the canonical relations as before, while the second equality uses the induction hypothesis of $k$ diffractions. This finishes the proof of the lemma.
\end{proof}

\section{Regularization and microlocal partition of unity of trace}
\label{reg&part}
This section involves some preparations for computing the leading order singularities of the wave trace $``\Tr U_{\aalpha}(t)"$ near $t=L$ where $L$ is the length of a closed diffractive geodesic. In particular, we shall develop a regularization of the wave trace and a microlocal partition of unity. 

\subsection{Trace regularization}
As mentioned before, although $U_{\aalpha}(t)$ is not even in trace class in the distributional sense, we can introduce a ``free" propagator $U_{0}(t)$ such that $U_{\aalpha}(t)-U_{0}(t)$ is indeed in the trace class in the sense of distributions as in \cite{sjostrand1997trace}. Intuitively, the role of the ``free" propagator is to cancel out both the long-range effect generated by the vector potential and the singularities at the diagonal due to the non-compactness. In order to achieve this goal, we choose a Hamiltonian $P_0$ with the vector potential 
\begin{equation}
    \Vec{A}_{\tilde{\alpha}}=-\tilde{\alpha}\cdot\left(-\frac{y-c_y}{(x-c_x)^2+(y-c_y)^2}, \frac{x-c_x}{(x-c_x)^2+(y-c_y)^2} \right)
\end{equation}
with $\tilde{\alpha}=\sum_{i=1}^n \alpha_i$ being the total flux and $c=(c_x,c_y)=\frac{1}{\tilde{\alpha}} \sum_{i=1}^n \alpha_i\cdot (x_i,y_i)$ being the ``center of mass" of all solenoids if $\tilde{\alpha}\neq0$. In complex coordinates, such a vector potential satisfies 
\begin{equation*}
\begin{split}
    \Vec{A}_{\tilde{\alpha}}-\sum_i \Vec{A}_i
    &=\frac{\tilde{\alpha}}{i}\cdot \frac{z-c}{|z-c|^2}-\sum_{i=1}^n\frac{\alpha_i}{i}\cdot\frac{z-z_i}{|z-z_i|^2}\\
    &= \frac{\tilde{\alpha}}{i}\cdot \frac{z}{|z|^2}\cdot\left(1+\frac{\bar{c}}{\bar{z}}+\cdots\right)-\sum_{i=1}^n \frac{\alpha_i}{i}\cdot \frac{z}{|z|^2}\cdot\left(1+\frac{\bar{z}_i}{\bar{z}}+\cdots\right)\\
    &= \mathcal{O}(r^{-3}).
\end{split}
\end{equation*}
Therefore, by \cite[Proposition 2.1]{sjostrand1997trace} we conclude that $U_{\aalpha}(t)-U_{0}(t)$ is in the trace class in the sense of distributions, where $U_0(t)\df e^{it\sqrt{P_0}}$.

Choose $R>0$ large enough such that all solenoids $s_i$ are enclosed in $B(0,R/3)$. Take $\chi\in\CI_c(\RR^2)$ such that $\chi\restrictedto_{B(0,2R/3)}\equiv 1$ and $\supp \chi \subset B(0,R)$. Note that $U_{\aalpha}(t)-U_{0}(t)$ is not a well-defined operator in general since $U_{\aalpha}(t), U_{0}(t)$ act on different domains. To make sense of this, the correct definition of the renormalized trace is given by
\begin{equation}
\label{dfn_trace}
\begin{split}
     \Tr \big[ U_{\aalpha}(t)-U_{0}(t)\big] \df & \Big[\Tr \big[\chi U_{\bullet}(t)\chi\big]+\Tr \big[(1-\chi) U_{\bullet}(t)\chi\big] + \Tr \big[\chi U_{\bullet}(t)(1-\chi)\big]\Big]\Big\rvert_{0}^{\aalpha} \\
     & + \Tr \big[ (1-\chi) U_{\bullet}(t)(1-\chi)\restrictedto_{0}^{\aalpha}\big],
\end{split}
\end{equation}
since when restricted to $\RR^2\backslash B(2R/3)$ the two domains are the same. Since all closed diffractive geodesics are enclosed in $\supp \chi$, it is natural to consider the trace of $U_{\aalpha}(t)$ restricted to this compact set.  

\begin{lemma}
\label{lemma_trace_identity}
Given the ``free" propagator $U_{0}(t)$ and $\chi\in\CI_c(\RR^2)$ defined above,  
\begin{equation}
    \Tr \big[ U_{\aalpha}(t)-U_{0}(t)\big] - \Tr \big[\chi U_{\aalpha}(t) \chi\big] \in \CI. 
\end{equation}
\end{lemma}

\begin{proof}
By the definition of the renormalized trace \eqref{dfn_trace}, we need to show all the terms in equation \eqref{dfn_trace} except $\Tr \big[\chi U_{\aalpha}(t) \chi\big]$ are smooth in $t$. This is an application of Lemma \ref{lemma_pgtr_equiv} and cyclicity of the trace. 

We only show that $\Tr \big[\chi U_{0}(t) \chi\big]$ is smooth; proofs for the rest of the terms are similar but easier, since no closed diffractive geodesic intersects $\supp 1-\chi$. We take a finite microlocal partition of unity $\{\Pi_l\}\in\Psi^0_c(X)$, $\psi\in \CI_c(X)$ of $\supp \chi$ such that
\begin{itemize}
    \item $\WF'\Pi_l$ is small enough with another set of pseudodifferential operators $A_l$ such that $\Pi_l-A_l^2\in\Psi^{-\infty}_c(X)$. $\psi$ has smooth square root. 
    \item $\supp\psi\subset B(c,\delta)$ and $\psi\rvert_{B(c,\delta/2)}\equiv1$ for some $\delta$ small enough. 
    \item $\psi$ and $\{\Pi_l\}$ form a complete microlocal partition of unity in the sense that
    $$\WF'\left(\Id-\psi-\sum_{l} \Pi_l\right)\subset T^*(\RR^2\backslash\supp\chi).$$
\end{itemize} 
Therefore, 
\begin{equation}
\begin{split}
     \Tr \big[\chi U_{0}(t) \chi\big] 
     &\equiv \Tr \big[\chi U_{0}(t) \chi (\sum_l \Pi_l)\big]+\Tr \big[\chi U_{0}(t) \chi \psi\big] \\
     &\equiv \sum_l \Tr \big[ A_l \chi U_{0}(t) \chi  A_l \big]+ \Tr \big[\psi U_{0}(t) \psi\big] \mod \CI.
\end{split}
\end{equation}
By part (1) of Lemma \ref{lemma_pgtr_equiv}, since for any $t\in\RR$ there is no closed diffractive geodesic,
\begin{equation}
    A_l \chi U_{0}(t) \chi A_l:  \dom^{-\infty} \longrightarrow \dom^{\infty}
\end{equation}
for every $A_l$, where $\dom$ is the domain of $P_{0}$. Hence, taking trace of $A_l \chi U_{0}(t) \chi A_l$ yields a smooth function on $\RR$. Similarly, by Remark \ref{rmk_solenoid_cutoff} taking trace of $\psi U_{0}(t) \psi$ also yields a smooth function on $\RR$ since there are no diffractive geodesics starting and ending in $\supp \psi$ for any $t$. 
\end{proof}

By Lemma \ref{lemma_trace_identity}, singularities of the regularized trace $\Tr \big[ U_{\aalpha}(t)-U_{0}(t)\big]$ are the same as singularities of the localized trace $\Tr \big[\chi U_{\aalpha}(t) \chi\big]$. For this reason, we shall study the latter in the following sections.

\subsection{Microlocal partition of unity}

In order to reduce the computation of $\Tr \big[\chi U_{\aalpha}(t) \chi\big]$ to microlocalized propagators, we construct a microlocal partition of unity of $\supp \chi$ carefully. We start by choosing $\psi_i\in \CI_c(X), 1\leq i\leq n$ to be smooth cutoff functions at each solenoid such that $\psi_i\equiv 1$ near a $\delta/2$-neighborhood of the solenoid $s_i$ and whose support is contained in a $\delta$-neighborhood of $s_i$. In particular, we also require that each $\psi_i$ has a smooth square root for later use. Multiplying by $\psi_i$ therefore localizes within $\delta$-neighborhood of the solenoid $s_i$. $\{\psi_i\}_{1\leq i\leq n}$ are called \emph{solenoid cutoffs}. Next, let $\{\Pi_k\}_{1\leq k\leq N}\subset\Psi^0_c(X^{\circ})$ be a finite collection of pseudodifferential operators on the interior of $X$ satisfying the following properties: 
\begin{enumerate}
    \item The Schwartz kernel of each $\Pi_k$ has compact support; 
    \item There exists a fixed constant $\delta_{\text{int}}$ such that  $\WF'\Pi_k\subset S^*X$ is contained in a small ball of radius $\delta_{\text{int}}$ with respect to a fixed Finsler metric on $S^*X$;
    \item The $\Pi_k$'s complete the solenoid cutoffs to a microlocal partition of unity on $\supp \chi$ in the sense that 
    \begin{equation*}
        \WF'\left(\Id-\sum_{i=1}^n\psi_i-\sum_{k=1}^N\Pi_k\right)\subset T^*(\RR^2\backslash\supp\chi); 
    \end{equation*}
    \item Each $\Pi_k$ has a square root $\sqrt{\Pi_k}\in\Psi^0_c(X^{\circ})$ modulo smoothing error, i.e., $\Pi_k-(\sqrt{\Pi_k})^2\in \Psi^0_c(X^{\circ})$. The solenoid cutoffs $\{\psi_i\}$ also have smooth square roots.
\end{enumerate}
These pseudodifferential operators $\{\Pi_k\}$ are called \emph{interior microlocalizers}. Note that by adjusting the constant $\delta$ and $\delta_{\text{int}}$, we may choose the solenoid cutoffs/interior microlocalizers as small as we want while still being a finite family. We define $\{B_j\}_{1\leq j\leq n+N}\df \{\Pi_k\}\cup \{\psi_i\}$ to be the complete set of microlocal partition of unity of $\supp \chi$.

For later use, we need to further refine our microlocal partition of unity such that it has the following property: 
\begin{property1}
\item Supposed that $\psi_i$ and $\Pi_k$ are a solenoid cutoff and a interior microlocalizer in the above microlocal partition of unity. For any $q\in \WF'\Pi_k$, consider all points $p\in \WF' \psi_i\subset S^*X$ such that $p$ and $q$ are diffractively related by any fixed time $t_0$, then for each fixed $t_0>0$ either 
    \begin{enumerate}
        \item the projection to the base of $p$: $p^{\mathfrak{b}}\in\{\psi_i\equiv 1\}$ for any $q\in \WF'\Pi_k$, or 
        \item $d(p^{\mathfrak{b}},s_i)>\frac{1}{10}\delta$ for any $q\in \WF'\Pi_k$.
    \end{enumerate}
\end{property1}
This property ensures that the diffractive geodesic flowout of such $\WF'(\Pi_k)$ at any fixed time $t_0$ either stays slightly away from the solenoid, or lies close to the solenoid within the set $\{\psi_i\equiv 1\}$. 

Note that this property may be ensured by leaving $\psi_i$ fixed and shrinking the microsupport of $\Pi_k$ as needed. We take $\delta_{\text{int}}<\frac{1}{100}\delta$ in the definition of the microlocal partition of unity such that all $\WF'\Pi_k$ is contained in a small ball of radius $\delta_{\text{int}}$ with respect to a fixed Finsler metric on $S^*X$. If there exist one point $q\in \WF'\Pi_k$ such that its time $t_0$ flowout enters $B(s_i,\frac{1}{5}\delta)$, then time $t_0$ flowout of all points in $\WF'\Pi_k$ are contained in $\{\psi_i\}\equiv 1$, so partition property (1) is satisfied. Otherwise, the time $t_0$ flowout of all points in $\WF'\Pi_k$ are strictly outside $B(s_i,\frac{1}{5}\delta)$. Therefore it satisfies the partition property (2).

\section{Trace of the Aharonov--Bohm wave propagator}

With all the tools established in previous sections, we now in a position to prove the main theorem on singularities of the wave trace.

\subsection{Trace decomposition}
Using the microlocal partition of unity of $\supp\chi$ constructed in the previous section, for each $t\in\RR$:
\begin{equation*}
    \chi U_{\aalpha}(t) \chi- \chi U_{\aalpha}(t) \chi \left(\sum_{j=1}^{n+N}B_j\right):\ 
    \dom_{\aalpha}^{-\infty} \longrightarrow \dom_{\aalpha}^{+\infty}.
\end{equation*}
Hence taking the trace of this operator yields a smooth function on $\RR$. Thus the singularities of $\Tr \big[\chi U_{\aalpha}(t) \chi\big]$ are the same as the sum over $\{1\leq j\leq n+N\}$ of the singularities of 
\begin{equation}
\label{ML_trace}
 \Tr  \big[\chi U_{\aalpha}(t) \chi B_j\big].   
\end{equation}
Given a fixed $j$, $B_j$ can be either a solenoid cutoff $\psi_i$ or an interior microlocalizer $\Pi_k$. 

In order to see when such a term contributes to singularities of the trace, using cyclicity of the trace, \eqref{ML_trace} is the same as 
$$\Tr \big[\sqrt{B_j} \chi U_{\aalpha}(t) \chi \sqrt{B_j}\big].$$
Consider the propagator $\sqrt{B_j} \chi U_{\aalpha}(t) \chi \sqrt{B_j}$ near $t=L$ for a fixed $L>0$. Then
\begin{equation}
\label{mleq}
\sqrt{B_j} \chi U_{\aalpha}(t) \chi \sqrt{B_j}\equiv \sqrt{B_j} \chi U_{\aalpha,\gamma}(t) \chi \sqrt{B_j} \mod   \dom_{\aalpha}^{-\infty}\longrightarrow\dom_{\aalpha}^{\infty}
\end{equation}
for a diffractive geodesic $\gamma$ with $\gamma(0),\gamma(t)\in \WF'B_j$ by Lemma \ref{lemma_pgtr_equiv}\footnote{We can also generalize Lemma \ref{lemma_pgtr_equiv} to solenoid cutoffs here.}. 
For any $L>0$ \emph{not} being the length of a closed diffractive geodesic, we may refine the microlocal partition by shrinking $\WF'B_j$ such that there is no diffractive geodesic $\gamma$ with $\gamma(0), \gamma(t)\in \WF'B_j$ for $t>0$ near $L$. This can be achieved by taking $\delta_{\text{int}}$ small enough. Therefore, the RHS of the operator equation \eqref{mleq} is $\dom_{\aalpha}^{-\infty}\longrightarrow\dom_{\aalpha}^{\infty}$ in view of part (1) of Lemma \ref{lemma_pgtr_equiv}. Taking trace thus yields a smooth function near $t=L$. This shows the contribution to singularities of the trace near $t=L$ comes from closed diffractive geodesics at length $L$. Moreover, if there are no closed diffractive geodesics of length $L$, all contribution to the total trace from \eqref{ML_trace} is smooth; $\Tr \big[\chi U_{\aalpha}(t) \chi\big] $ is smooth near $t=L$. 

For convenience, we further assume that \emph{$\gamma$ is the only\footnote{For the situation of multiple diffractive geodesics with the same length $L$, we add up the contribution from each one to yield the total singularities at $t=L$.} closed diffractive geodesic with length $L$}. Therefore, in order to calculate the trace near $t=L$, we only need to consider the solenoid cutoffs and interior microlocalizers $B_j$ such that $\gamma\cap \WF'B_j\neq\emptyset$. We use $\{B_{j_l}\}_{1\leq l\leq M}$ to denote the set of all such $B_j$. In particular, $\{B_{j_l}\}_{1\leq l\leq M}$ forms a microlocal partition of unity near $\gamma$ since it is a subset of the microlocal partition of unity $\{B_j\}$ of $\supp\chi$. 

Therefore, near $t=L$, the previous discussion yields the following decomposition:
\begin{equation}
\label{trace_decomp}
    \begin{split}
        \Tr \big[\chi U_{\aalpha}(t) \chi \big]
        & \equiv \Tr \left[\chi U_{\aalpha}(t) \chi \left(\sum_{l=1}^{M}B_{j_l}\right)\right]\mod\CI \\
        & = \sum_{\psi_i\in\{B_{j_l}\}} \Tr \big[\chi U_{\aalpha}(t) \chi \psi_i\big] + \sum_{\Pi_k\in\{B_{j_l}\}} \Tr \big[\chi U_{\aalpha}(t) \chi \Pi_k\big]\\
        & = \sum_{\psi_i\in\{B_{j_l}\}} \Tr \big[U_{\aalpha}(t) \psi_i\big] + \sum_{\Pi_k\in\{B_{j_l}\}} \Tr \big[U_{\aalpha}(t) \Pi_k\big]\\
        & \equiv \sum_{\psi_i\in\{B_{j_l}\}} \Tr \big[U_{\aalpha,\gamma}(t) \psi_i\big] + \sum_{\Pi_k\in\{B_{j_l}\}} \Tr \big[U_{\aalpha,\gamma}(t) \Pi_k\big] \mod \CI.\\
    \end{split}
\end{equation}
The third equality is due to cyclicity of the trace and the fact $\chi\psi_i=\psi_i$, $\chi\Pi_k=\Pi_k$ for $\psi_i,\Pi_k\in \{B_{j_l}\}_{1\leq l\leq M}$; whereas the last equality follows from Lemma \ref{lemma_pgtr_equiv} and the remark after that. More precisely, $\psi_i$ and $\Pi_k$ in the sum form a microlocal partition of unity near the closed diffractive geodesic $\gamma$. It suffices to consider these two types of contributions to the trace. Note that if we can also reduce the first type $\Tr \big[U_{\aalpha,\gamma}(t) \psi_i\big]$ to the representation
$$\Tr \big[\sqrt{\Pi} U_{\aalpha,\gamma}(t) \sqrt{\Pi}\big] \equiv \Tr \big[ U_{\aalpha,\gamma}(t) \Pi\big]\mod \CI $$
where $\Pi$ is an interior microlocalizer, then we may compute the trace of the resulting term using Proposition \ref{prop_pgtr2}.

\subsection{Microlocalized propagators with solenoid cutoffs}
Consider $\Tr \big[U_{\aalpha,\gamma}(t) \psi_i\big]$ for $\psi_i\in \{B_{j_l}\}_{1\leq l\leq M}$ in the equation \eqref{trace_decomp}. In particular, the $\psi_i$'s appearing here are exactly the solenoid cutoffs at each diffraction along the closed diffractive geodesic $\gamma$. Choose a short time $t_1$ such that $t_1<\frac{1}{10}\min_{1\leq j\leq k\leq n}d(s_j,s_k)$; for such $t_1$, the backward propagation of $\psi_i$ from $s_i$, which are all points in $S^*X$ arrive at $\WF' \psi_i$ under time $t_1$ diffractive geodesic flow defined in Section \ref{DG}, stays away from any solenoid. Writing $t_2\df t-t_1$, by the group property of $U_{\aalpha}(t)$,
\begin{equation*}
\begin{split}
    \Tr \big[U_{\aalpha,\gamma}(t) \psi_i\big] \equiv \Tr \big[ U_{\aalpha}(t) \psi_i \big]
    &= \Tr \big[ U_{\aalpha}(t_1) U_{\aalpha}(t_2) \psi_i \big].
\end{split}
\end{equation*}
Since $\{B_{j_l}\}$ is a partition of unity near $\gamma$,
\begin{equation*}
\begin{split}
   U_{\aalpha}(t_1) U_{\aalpha}(t_2) \psi_i 
  \equiv \sum_{\Pi_k\in\{B_{j_l}\}} U_{\aalpha}(t_1) \Pi_k U_{\aalpha}(t_2) \psi_i \mod \dom_{\aalpha}^{-\infty}\longrightarrow\dom_{\aalpha}^{\infty}.
\end{split}
\end{equation*}
Taking trace of the above equation, using cyclicity of the trace we obtain
\begin{equation*}
\begin{split}
    \Tr \big[ U_{\aalpha}(t) \psi_i \big]
    \equiv \sum_{\Pi_k\in\{B_{j_l}\}} \Tr \big[ U_{\aalpha}(t_2) \psi_i U_{\aalpha}(t_1) \Pi_k \big] \mod \CI.
\end{split}
\end{equation*}
By the Partition Property 1, for each $\Pi_k$, the time $t_1$ flowout of $\WF'{\Pi_k}$ is either contained in $\{\psi_i\equiv 1\}$ or stays slightly away from the solenoid. Now we consider these two cases individually.

\emph{Case 1.}
We first assume that $\psi_i, \Pi_k$ satisfy part (1) of the Partition Property 1, i.e., the projection to $X$ of the time $t_1$ flowout of $\WF'{\Pi_k}$ along $\gamma$ is contained in $\{\psi_i\equiv 1\}$, we may replace $\psi_i$ by the identity operator:
\begin{equation*}
    U_{\aalpha}(t_2) \psi_i U_{\aalpha}(t_1) \Pi_k
    \equiv U_{\aalpha}(t_2) \Id U_{\aalpha}(t_1) \Pi_k \mod \dom_{\aalpha}^{-\infty}\longrightarrow\dom_{\aalpha}^{\infty}.
\end{equation*}
Taking the trace yields of the operators yields
\begin{equation*}
\begin{split}
   \Tr \big[U_{\aalpha}(t_2) \psi_i U_{\aalpha}(t_1) \Pi_k\big]
   & \equiv \Tr \big[U_{\aalpha}(t) \Pi_k\big] \equiv \Tr \big[U_{\aalpha,\gamma}(t) \Pi_k\big] \mod\CI
\end{split}
\end{equation*}
where the last equality is due to Lemma \ref{lemma_pgtr_equiv}. We therefore have reduced the first situation to the desired form. 

\emph{Case 2.}
Next, we assume that $\psi_i, \Pi_k$ satisfy part (2) of the Partition Property 1, i.e., the time $t_1$ flowout of $\WF'\Pi_k$ stays away from the solenoid. Note that the diffraction at $s_i$ can happens either within the propagation of $U_{\aalpha}(t_2)$ or $U_{\aalpha}(t_1)$. Here we only discuss the first scenario: free propagation of $U_{\aalpha}(t_1)$, which means the projection of the time $t$ flowout of $\WF'\Pi_k$ to the base does not contain any element of $S$ for $t\in[0,t_1]$; the second scenario can be reduced to the first one using cyclicity of the trace. We apply Egorov's theorem (cf. \cite[Section 11.1]{zworski2012semiclassical}) to pull-back the contribution from $\psi_i$, so that it looks exactly like in the first case. Using the group property of $U_{\aalpha}(t)$,
\begin{equation*}
\begin{split}
   \Tr \big[U_{\aalpha}(t_2) \psi_i U_{\aalpha}(t_1) \Pi_k\big]
   & = \Tr \big[U_{\aalpha}(t_2) \psi_i  [U_{\aalpha}(t_1)\Pi_k U_{\aalpha}(-t_1)] U_{\aalpha}(t_1)\big]\\
   & = \Tr \big[U_{\aalpha}(t_2) \psi_i  \widetilde{\Pi}_k  U_{\aalpha}(t_1)\big]
\end{split}
\end{equation*}
where $\widetilde{\Pi}_k\df U_{\aalpha}(t_1)\Pi_k U_{\aalpha}(-t_1)$ is a pseudodifferential operator in $\Psi^0_c(X^{\circ})$ with principal symbol $(G^{-t_1})^*a_k$, where $a_k$ is the principal symbol of $\Pi_k$. Therefore, $\psi_i \cdot \widetilde{\Pi}_k$ is an pseudodifferential operator in $\Psi^0_c(X^{\circ})$ with the principal symbol $\psi_i\cdot (G^{-t_1})^*a_k$. Applying Egorov's theorem once again yields
\begin{equation*}
\begin{split}
\Tr \big[U_{\aalpha}(t_2) \psi_i  \widetilde{\Pi}_k  U_{\aalpha}(t_1) \big]
& = \Tr \big[U_{\aalpha}(t_2+t_1) \big( U_{\aalpha}(-t_1) \psi_i  \widetilde{\Pi}_k  U_{\aalpha}(t_1) \big)\big]\\
& = \Tr \big[U_{\aalpha}(t)  \Pi_{k,\psi_i}\big]
\end{split}
\end{equation*}
where $\Pi_{k,\psi_i}\df  U_{\aalpha}(-t_1) \psi_i  \widetilde{\Pi}_k  U_{\aalpha}(t_1)$ is a pseudodifferential operator in $\Psi^0_c(X^{\circ})$ with the principal symbol $(G^{t_1})^*\big(\psi_i\cdot (G^{-t_1})^*a_k\big)$. By Lemma \ref{lemma_pgtr_equiv}, 
$$\Tr \big[U_{\aalpha}(t)  \Pi_{k,\psi_i}\big] 
\equiv \Tr \big[U_{\aalpha,\gamma}(t) \Pi_{k,\psi_i}\big] \mod \CI$$  
for a diffractive geodesic $\gamma$ with $\gamma(0),\gamma(t)\in \WF'\Pi_{k,\psi_i}$. Therefore, it suffices to compute the trace of the microlocalized propagation $U_{\aalpha,\gamma}(t) \Pi_{k,\psi_i}$. 

\subsection{Trace of a microlocalized propagator}
In the previous subsection, we have reduced each term in equation \eqref{trace_decomp} to the trace of the corresponding microlocalized propagator along the closed diffractive geodesic $\gamma$. The purpose of this subsection is to compute the microlocalized trace $\Tr [U_{\aalpha,\gamma}(t)\Pi]$ for some interior microlocalizer $\Pi$ near $\gamma$. 

Recall the closed diffractive geodesic $\gamma$ has total length $L$ and $m$ diffractions at points $p_1,\cdots,p_m$. Therefore $\gamma$ can be denoted by $m$ regular geodesics $\{\gamma_i\}_{1\leq i\leq m}$ (straight line segments) concatenated to a closed loop with each segment having length $l_1,\cdots,l_m$ respectively. 

We now perform the method of stationary phase in the base variable $z$ for the microlocalized propagator \eqref{int_pgtr2} in Proposition \ref{prop_pgtr2} restricted to the diagonal and integrated in $X$. In particular, here we take $\Pi_s=\Pi_e=\sqrt{\Pi}$ by cyclicity of the trace. Note that unlike in Proposition \ref{prop_pgtr2}, we do not using stationary phase in the phase variable $\lambda$ here. 
\begin{equation}
 \Tr \big[ U_{\aalpha,\gamma}(t) \Pi\big]=
 \int_X \int_{\RR}e^{i\varphi}b(t,\bar{r}_{z},r_{z},\bar{\theta}_{z},\theta_{z};\lambda)d\lambda dz   
\end{equation}
where the phase function $\varphi=\lambda\left(t-r_{z}-\sum_{i=1}^{m-1}l_i-\bar{r}_{z}\right)$ and the symbol $b\in S_c^{-\frac{m-1}{2}}$ is given by 
\begin{equation}
    (2\pi)^{\frac{m-3}{2}} \cdot e^{-i\frac{(m-1)\pi}{4}} \cdot \frac{d_{\gamma,\alpha}}{l_{\gamma,\alpha}^{\frac{1}{2}}} \cdot b_{\Pi}(z;\lambda\theta_{z}) \cdot \rho(\lambda) \cdot \lambda^{-\frac{m-1}{2}} \mod S_c^{-\frac{m}{2}} 
\end{equation}
where $b_{\Pi}(z;\xi)$ is the amplitude of $\Pi$. Note that the variable $z$ is supported in a compact subset of $X$ owing to the support of the amplitude. Also recall that $(r_z,\theta_z)$ and $(\bar{r}_{z},\bar{\theta}_{z})$ are the coordinates of $z$ under the polar coordinates centered at the first and last diffractions correspondingly.  

The phase function $\varphi$ is critical in $z$-variable precisely when 
\begin{equation*}
    \partial_z(r_z+\bar{r}_{z})=0.
\end{equation*}
Then as before this forces $z$ to lie along the line segment $\gamma_m^{\mathfrak{b}}$ of length $l_m$ connecting the solenoids where the first and the last diffractions happens. Note that there is no integration in the phase variables, unlike the previous calculation in Proposition \ref{prop_pgtr2}; as a result there is no particular point along $\gamma_m$ which is fixed by stationarity. Therefore, under the polar coordinates $(r_z,\theta_z)$, it suffices to consider stationary phase in the $\theta_z$-variable and the integration in $r_z$-variable.   

Compute the Hessian in $\theta_z$: 
\begin{equation*}
    \varphi''(\theta_z)=-\frac{r_z\cdot l_m}{\bar{r}_z}\restrictedto_{r_z+\bar{r}_z=l_n}.
\end{equation*}
The method of stationary phase yields that the microlocalized trace $\Tr \big[ U_{\aalpha,\gamma}(t) \Pi\big]$ has the oscillatory integral expression
\begin{equation}
        \int_{\RR_{\lambda}} e^{i\lambda\left(t-L \right)} a_{\Pi}(t;\lambda)d\lambda 
\end{equation}
where the symbol $a_{\Pi}\in S_c^{-\frac{m}{2}}(\RR; \RR_{\lambda})$ (cf. equation \eqref{symb_pgtr2}) is
\begin{equation}
    \begin{split}
        a_{\Pi} 
        &\equiv \int_{r_z=0}^{l_m} (2\pi)^{\frac{1}{2}}\cdot e^{-\frac{i\pi}{4}}\cdot\left(\frac{\bar{r}_z}{r_z\cdot l_m}\right)^{\frac{1}{2}}\cdot \rho(\lambda) \cdot \lambda^{-\frac{1}{2}}\cdot b(t,\bar{r}_{z},r_{z},\bar{\theta}_{z},\theta_{z};\lambda)\cdot r_z\restrictedto_{z\in\gamma_m} dr_z\\
        &\equiv \left[ (2\pi)^{\frac{m-2}{2}} \cdot e^{-i\frac{m\pi}{4}} \cdot \frac{d_{\gamma,\aalpha}}{{\prod_{i=1}^{m}l_i}^{\frac{1}{2}}} \cdot \rho(\lambda) \cdot \lambda^{-\frac{m}{2}} \right] \cdot \left(\int_{r_z=0}^{l_m} b_{\Pi}(z;\lambda\theta_{z})\restrictedto_{z\in\gamma_m} dr_{z}\right) \mod S_c^{-\frac{m+1}{2}}. 
    \end{split}
\end{equation}
Note that since the product of the smooth cutoff functions in $\lambda$ in the first line satisfies the same condition, we shall use the same notation $\rho(\lambda)$ to denote it for convenience. Here it is important to notice that the last integral of $b_{\Pi}(z;\lambda\theta_{z})\restrictedto_{z\in\gamma_m}$ in $r_z$ only depends on the amplitude of $\Pi$ restricted to $\gamma_m$, which is exactly the location of all stationary points. Also, we can extend the limit of integration to the whole diffractive geodesic $\gamma$ since $b_{\Pi}$ is supported only on $\gamma_m$. We define 
\begin{equation}
\label{PS}
  a_0\df (2\pi)^{\frac{m-2}{2}} \cdot e^{-i\frac{m\pi}{4}} \cdot \frac{d_{\gamma,\aalpha}}{{\prod_{i=1}^{m}l_i}^{\frac{1}{2}}} \cdot \rho(\lambda) \cdot \lambda^{-\frac{m}{2}}. 
\end{equation}
It is also important to notice that $a_0$ is independent of the interior microlocalizer $\Pi$. 

\subsection{Assembling the pieces}
In this subsection, we assemble the contribution to the trace from each piece of microlocalized propagator in equation \eqref{trace_decomp}; this yields the total singularities of $\Tr \chi U_{\aalpha}(t) \chi$ near $t=L$, where $L$ is the length of the closed diffractive geodesic $\gamma$.

Recall that by equation \eqref{trace_decomp}
\begin{equation*}
    \begin{split}
        \Tr \big[\chi U_{\aalpha}(t) \chi \big]
        \equiv \sum_{\psi_i\in\{B_{j_l}\}} \Tr \big[U_{\aalpha,\gamma}(t) \psi_i\big] + \sum_{\Pi_k\in\{B_{j_l}\}} \Tr \big[U_{\aalpha,\gamma}(t) \Pi_k\big] \mod \CI
    \end{split}
\end{equation*}
where $\{B_{j_l}\}_{1\leq l\leq M}\subset\{B_j\}_{1\leq l\leq n+N}$ is a microlocal partition of unity near $\gamma$. The discussion in Section 6.2 further reduce the terms in the first summation into microlocalized propagators
\begin{equation*}
    \begin{split}
        \Tr \big[\chi U_{\aalpha}(t) \chi\big] 
        \equiv \sum_{\psi_i\in\{B_{j_l}\}} \sum_{1\leq k\leq N} \Tr \big[U_{\aalpha,\gamma}(t) \tilde{\Pi}_{k,\psi_i}\big] + \sum_{\Pi_k\in\{B_{j_l}\}} \Tr \big[U_{\aalpha,\gamma}(t) \Pi_k\big] \mod \CI
    \end{split}
\end{equation*}
where $\tilde{\Pi}_{k,\psi_i}=\Pi_k$ if $\Pi_k$ satisfies Partition Property 1(1); otherwise $\Pi_k$ satisfies Partition Property 1(2) and $\tilde{\Pi}_{k,\psi_i}=\Pi_{k,\psi_i}$. We need the following lemma: 

\begin{lemma}
\label{Loc_trace}
The singularities of $\Tr \big[U_{\aalpha,\gamma}(t) \psi_i\big]$ near $t=L$ is given by
\begin{equation}
    \Tr \big[U_{\aalpha,\gamma}(t) \psi_i\big] = \int_{\RR_{\lambda}} e^{i\lambda\left(t-L \right)} a_{\psi_i}(t;\lambda)d\lambda 
\end{equation}
where the amplitude $a_{\psi_i}\in S_c^{-\frac{m}{2}}(\RR; \RR_{\lambda})$ is given by 
\begin{equation}
  a_{\psi_i}\equiv a_0\cdot  \left(\int_{z\in\gamma} \psi_i\restrictedto_{z\in\gamma} dr_{z}\right) \mod S_c^{-\frac{m+1}{2}}.
\end{equation}
\end{lemma}

\begin{proof}
The discussion in the Subsection 6.2 yields 
\begin{equation}
    \Tr \big[U_{\aalpha,\gamma}(t) \psi_i\big]\equiv \sum_{1\leq k\leq N} \Tr \big[U_{\aalpha,\gamma}(t) \tilde{\Pi}_{k,\psi_i}\big] \mod \CI
\end{equation}
where $\tilde{\Pi}_{k,\psi_i}=\Pi_k$ if $\Pi_k$ satisfies Partition Property 1(1); otherwise $\Pi_k$ satisfies Partition Property 1(2) and $\tilde{\Pi}_{k,\psi_i}=\Pi_{k,\psi_i}$. For those $\Pi_k$ satisfying Partition Property 1(1), the discussion in Section 6.3 yields
\begin{equation}
    \Tr \big[U_{\aalpha,\gamma}(t) \Pi_k\big]\equiv  \int_{\RR_{\lambda}} e^{i\lambda\left(t-L \right)} a_{\Pi_k}(t;\lambda)d\lambda 
\end{equation}
with
\begin{equation*}
 a_{\Pi_k}(t;\lambda) \equiv a_0\cdot \left(\int_{z\in\gamma} b_{\Pi_k}(z;\lambda\theta_{z})\restrictedto_{z\in\gamma} dr_z \right) \mod S_c^{-\frac{m+1}{2}} 
\end{equation*}
where $b_{\Pi_k}$ is the amplitude of $\Pi_k$. Since $\Pi_k$ satisfies Partition Property 1(1), the integral in the above equation 
\begin{equation}
\label{type1}
 \int_{z\in\gamma} b_{\Pi_k}(z;\lambda\theta_{z})\restrictedto_{z\in\gamma} dr_z 
 = \int_{z\in\gamma} \big(G^{t_1}\big)^*\psi_i\restrictedto_{\gamma} \cdot b_{\Pi_k}(z;\lambda\theta_{z})\restrictedto_{z\in\gamma} dr_z.
\end{equation}
Note that $G^t$ restricted to $\gamma$ is just the translation along $\gamma$ by $t$. Therefore, $\big(G^{t}\big)^*\psi_i\restrictedto_{\gamma}$ is the pull-back of $\psi_i\restrictedto_{\gamma}$ along $\gamma$ by $t$. Moreover, $dr_z$ is translation invariant along $\gamma$. 

On the other hand, if $\Pi_k$ satisfies Partition Property 1(2), the same argument applies although the symbol in this situation is given by Egorov's theorem as in the Subsection 6.2. More precisely, 
\begin{equation}
    \Tr \big[U_{\aalpha,\gamma}(t) \Pi_{k,\psi_i}\big]\equiv  \int_{\RR_{\lambda}} e^{i\lambda\left(t-L \right)} a_{\Pi_{k,\psi_i}}(t;\lambda)d\lambda 
\end{equation}
with
\begin{equation}
\label{type2}
 a_{\Pi_{k,\psi_i}}(t;\lambda) \equiv a_0\cdot \left(\int_{z\in\gamma} \big(G^{t_1}\big)^*\psi_i\restrictedto_{\gamma}\cdot b_{\Pi_k}(z;\lambda\theta_z)\restrictedto_{z\in\gamma} dr_z \right) \mod S_c^{-\frac{m+1}{2}} 
\end{equation}
Combining the equations \eqref{type1} and \eqref{type2} for all interior microlocalizers $\{\Pi_k\}$, it yields that 
\begin{equation*}
    \sum_k \int_{z\in\gamma} \big(G^{t_1}\big)^*\psi_i\restrictedto_{\gamma} \cdot b_{\Pi_k}(z;\lambda\theta_{z})\restrictedto_{z\in\gamma} dr_z 
    = \int_{z\in\gamma} \big(G^{t_1}\big)^*\psi_i\restrictedto_{\gamma} dr_z,
\end{equation*}
since $\{\Pi_k\}$ is a microlocal partition of unity away from the solenoids. The translation invariance of $dr_z$ on $\gamma$ near $s_i$ therefore yields
\begin{equation*}
    \int_{z\in\gamma} \big(G^{t_1}\big)^*\psi_i\restrictedto_{\gamma} dr_z=\int_{z\in\gamma} \psi_i\restrictedto_{\gamma} dr_z.
\end{equation*}
\end{proof}
Therefore, by the equation \eqref{trace_decomp} and the calculations of the (micro)localized trace in Subsection 6.3 and Lemma \ref{Loc_trace}, the amplitude $a(t;\lambda)$ of $\Tr \big[\chi U_{\aalpha}(t)\chi\big]$ near the singularity at $t=L$ is given by
\begin{equation}
    \sum_{\Pi_k\in\{B_{j_l}\}} a_0\cdot \left(\int_{z\in\gamma} b_{\Pi_k}(z;\lambda\theta_{z})\restrictedto_{z\in\gamma} dr_{z}\right) + \sum_{\psi_i\in\{B_{j_l}\}} a_0\cdot \left(\int_{z\in\gamma} \psi_i\restrictedto_{\gamma} dr_z\right)= a_0\cdot L_0 \mod S_c^{-\frac{m+1}{2}}
\end{equation}
where $L_0$ is the primitive length of the closed diffractive geodesic $\gamma$. Combining with the regularization Lemma \ref{lemma_trace_identity}, we therefore have proved the following theorem:  

\begin{theorem}
\label{main_thm}
Consider the wave propagator $U_{\aalpha}(t)=e^{-it\sqrt{P_{\aalpha}}}$ for the Hamiltonian defined in \eqref{TotHam}. The regularized wave propagator $U_{\aalpha}(t)-U_0(t)$ defined in Section \ref{reg&part} is in the trace class in distributional sense. The singularities of the regularized trace $\Tr \big[ U_{\aalpha}(t)-U_{0}(t)\big]$ defined by \eqref{dfn_trace} are given by lengths of all closed diffractive geodesics in $X$. In particular, the contribution to its singularity at $t=L$ that comes from the closed diffractive geodesic $\gamma$ with length $L$ is given by
\begin{equation}
\int_{\RR_{\lambda}} e^{i\lambda\left(t-L \right)} a(t;\lambda)d\lambda     
\end{equation}
where the principal symbol $a\in S^{-\frac{m}{2}}(\RR; \RR_{\lambda})$ is equal to
\begin{equation}
\label{prin_singularity}
    (2\pi)^{\frac{m-2}{2}} \cdot e^{-i\frac{m\pi}{4}} \cdot L_0 \cdot \frac{d_{\gamma,\aalpha}}{{\prod_{i=1}^{m}l_i}^{\frac{1}{2}}} \cdot \rho(\lambda) \cdot \lambda^{-\frac{m}{2}} \mod S^{-\frac{m+1}{2}}
\end{equation}
where $m$ is the number of diffractions; $\rho(\lambda)\in \CI(\RR_{\lambda})$ is a smooth function satisfying $\rho\equiv 0$ for $\lambda<0$ and $\rho\equiv 1$ for $\lambda>1$; $L_0$ is the primitive length of $\gamma$; $l_i$ is the length of the $i$-th piece geodesic of the diffractive geodesic $\gamma$ and 
\begin{equation}
\label{dga}
    d_{\gamma,\aalpha}= \prod_{i=1}^{m}d_{\alpha_i}(\beta_i)
\end{equation}
where $\alpha_i$ is the magnetic flux and $\beta_i$ is the diffraction angle at the $i$-th diffraction. 
\end{theorem}

\subsection{Fractional holonomy and the proof of the main theorem}

Now we use $l$ to label the order of the diffractions along $\gamma$ with $1\leq l\leq m$ and $k$ to label all the solenoids with $1\leq k\leq n$. Recall that $\{s_k\}_{k=1}^n$ and $\{\alpha_k\}_{k=1}^n$ are the sets of all solenoids and magnetic fluxes correspondingly; $\{s_{k_l}\}_{l=1}^m$, $\{\alpha_{k_l}\}_{l=1}^m$ and $\{\beta_{l}\}_{l=1}^m$ are correspondingly the solenoids, magnetic fluxes and diffraction angles along the diffractive geodesic $\gamma$ with $m$ diffractions. In particular, for each $l$, $s_{k_l}\in \{s_k\}_{k=1}^n$, $\alpha_{k_l}\in \{\alpha_k\}_{k=1}^n$ and there might exist $1\leq l<l'\leq m$ such that $s_{k_l}=s_{k_l'}$ and $\alpha_{k_l}=\alpha_{k_l'}$, i.e., there might be different diffractions happening at the same solenoid at different time along $\gamma$ and the diffraction angles may also be different (For example, at $s_1$ in Figure \ref{cdg}). 

\begin{figure}[H]
  \includegraphics[width=0.70\linewidth]{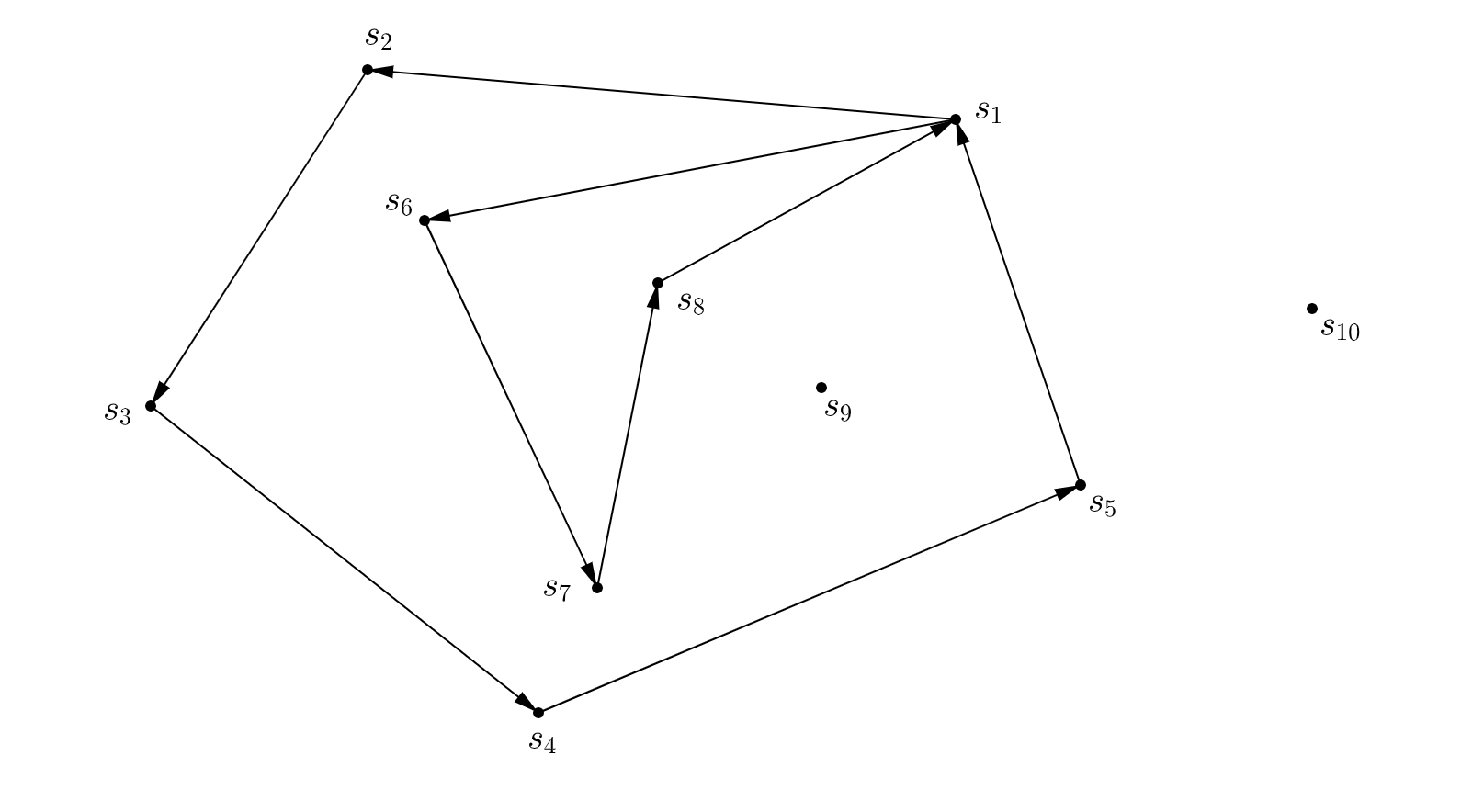}
  \caption{A closed diffractive geodesic $\gamma$}
  \medskip
  \small
  \label{cdg}
   
\end{figure}

We introduce a \emph{fractional winding number} $w_{\gamma,s_k}$ of $\gamma$ with respect to the solenoid $s_k$; the coefficient $d_{\gamma,\aalpha}$ indeed involves the fractional winding numbers of all the solenoids. 
\begin{definition}
\label{frac_winding}
For a closed diffractive geodesic $\gamma$, we define the \emph{fractional winding number} of $\gamma$ with respect to the solenoid $s_k$ to be 
\begin{equation}
    w_{\gamma,s_k}\df \emph{p.v. }\frac{1}{2\pi i}\oint_{\gamma}\frac{1}{z-z_k}dz,
\end{equation}
where \emph{p.v.} denotes the principal value of the integral and $z_k$ is the complex coordinate of $s_k$. 
\end{definition}
Therefore, we shall show later that combining equations \eqref{d-coeff} and \eqref{dga} yields
\begin{equation}
\label{frac_holonomy}
    d_{\gamma,\aalpha} = \left(\prod_{l=1}^{m} \sin(\pi\alpha_{k_l})\cdot \frac{e^{-i\left(\frac{\beta_{l}}{2}\right)}}{\cos(\frac{\beta_{l}}{2})}\right) \cdot \left(\prod_{k=1}^{n} e^{-2\pi i\cdot\alpha_k\cdot w_{\gamma,s_k}} \right)
\end{equation}
where the first term comes from the diffractions and the second term comes from the total holonomy of the closed diffractive geodesic $\gamma$. This suggests that the Aharonov--Bohm effect is actually embodied in principal symbols of singularities (cf. \cite{eskin2014aharonov}).

\begin{remark}
As special cases of fractional winding numbers, there are two basic scenarios (See Figure \ref{cdg2}). In fact, all closed diffractive geodesics can be reduced to a combination of these two scenarios. We assume that $\gamma$ is a simple closed curve.  
\begin{enumerate}
    \item the solenoid $s_k$ is enclosed in $\gamma$ or outside the the region which $\gamma$ enclosed, with no diffraction at it, then the fractional winding number is indeed the winding number of $\gamma$ (e.g. $s_1$ and $s_2$ in Figure \ref{cdg2});
    \item the solenoid $s_k$ is on $\gamma$ such that a diffraction happens at $s_k$ with diffraction angle $\beta$, then the fractional winding number equals to $-\beta/2\pi$ (cf. \cite{hungerbuhler2019non}). The negative sign is due to the fact that the definition of the diffraction angle and the angle in the contour integral differ by a sign (e.g. $s_3$ in Figure \ref{cdg2}).
\end{enumerate}
\end{remark}
\begin{figure}[H]
  \includegraphics[width=0.60\linewidth]{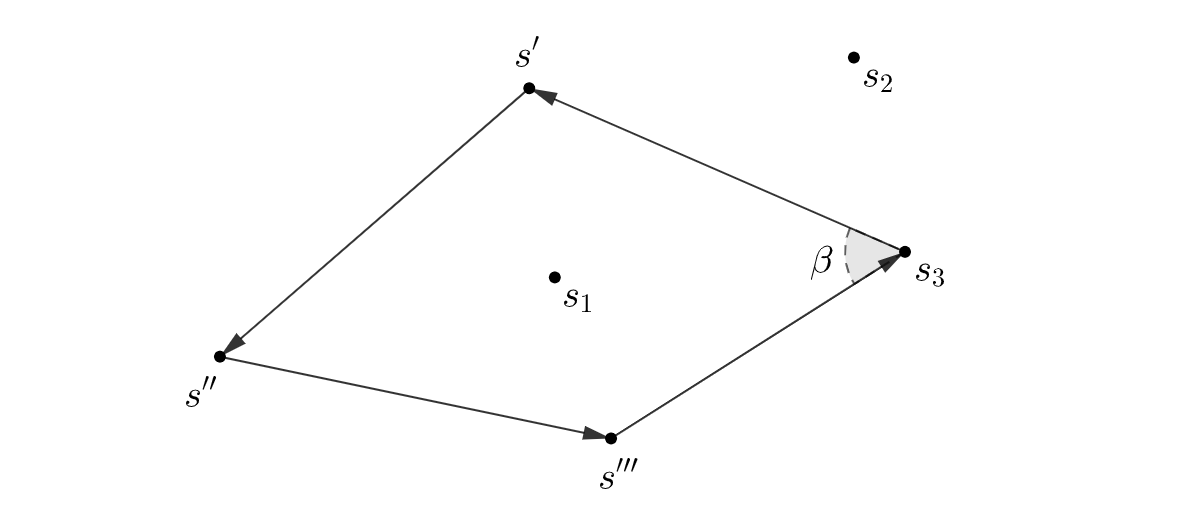}
  \caption{Fractional winding numbers in basic cases}
  \medskip
  \small
  \label{cdg2}
   Note that $s_1,s_2$ correspond to the first scenario mentioned in the above remark, while $s_3$ correspond to the second scenario. 
\end{figure}
To show \eqref{frac_holonomy} from \eqref{dga}, note that for a closed diffractive geodesic $\gamma$ with $m$ diffractions, we can break $\gamma$ into $m$ diffraction-pieces $\{\gamma_l\}_{1\leq l\leq m}$ with $s_{k_l}\in\gamma_l$ by adding $m$ points $\{p_l\}_{1\leq l\leq m}$ in the interior between each two consecutive diffractions (see Figure \ref{breakdown}). 
\begin{figure}[H]
  \includegraphics[width=0.75\linewidth]{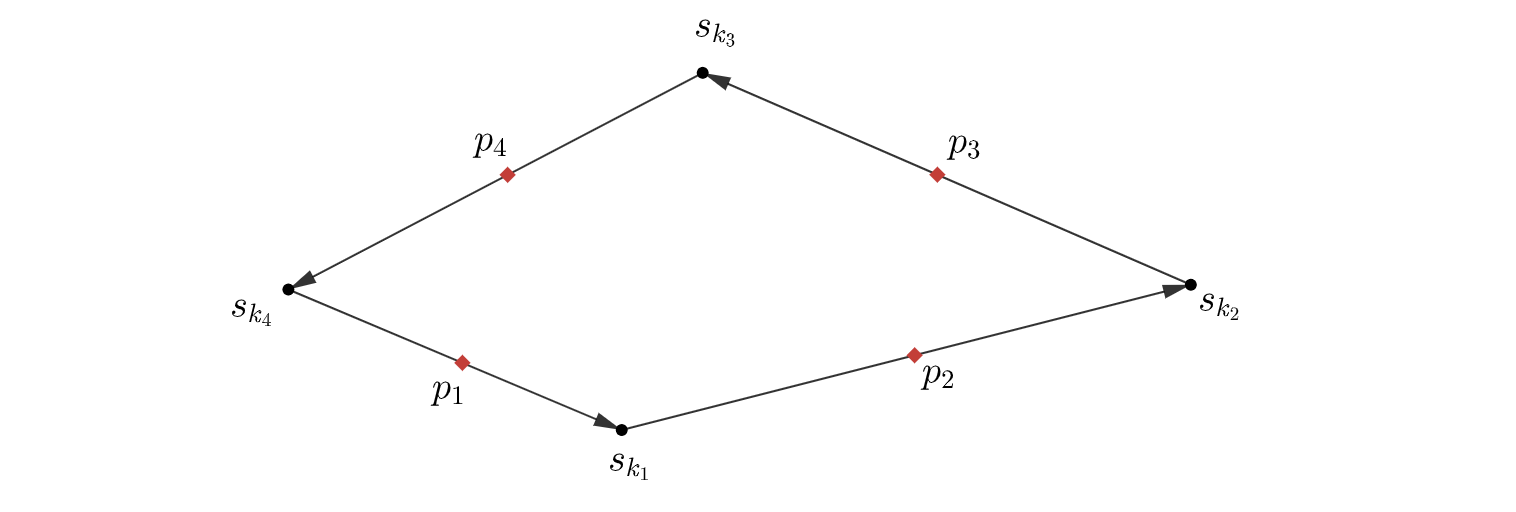}
  \caption{Inserting points between diffractions}
  \medskip
  \small
  \label{breakdown}
 
\end{figure}
In particular, $p_l$ is the point between $s_{k_{l-1}}$ and $s_{k_{l}}$; we define $s_{k_{0}}\df s_{k_m}$ and $p_{m+1}\df p_1$ since the trajectory is closed; $\gamma_l$ is defined to be the segment $p_{l}\rightarrow s_{k_l}\rightarrow p_{l+1}$. This process corresponds to taking microlocal cutoffs between each diffraction in Section \ref{sec_mlpg}. Therefore, we have 
\begin{equation}
    \label{holonomy_symbol}
    \begin{split}
        d_{\gamma,\aalpha} 
        &= \prod_{l=1}^{m}d_{\alpha_{k_l}}(\beta_{l})\\
        &=\left(\prod_{l=1}^{m} \sin(\pi\alpha_{k_l})\cdot \frac{e^{-i\left(\frac{\beta_{l}}{2}\right)}}{\cos(\frac{\beta_{l}}{2})}\right) 
        \cdot \left( \prod_{l=1}^{m} e^{-i\left(\sum_{j=1,j\neq {k_l}}^n\alpha_j\left(\phi_j(p_{l+1})-\phi_j(p_{l})\right)\right)} \right)
    \end{split}
\end{equation}
Consider the following lemma:
\begin{lemma}
\label{lemma_piece_integral}
For any decomposition $\{\gamma_l\}_{1\leq l\leq m}$ and $\{p_l\}_{1\leq l\leq m}$ of $\gamma$, we have 
\begin{equation}
    \label{piece_integral}
    \frac{1}{2\pi}\left(\phi_j(p_{l+1})-\phi_j(p_{l})\right)=\rea \left(\frac{1}{2\pi i}\int_{\gamma_l}\frac{1}{z-z_j}dz\right)
\end{equation}
if $s_{k_l}\neq s_j$, where $z_j$ is the complex coordinate of $s_j$.
\end{lemma}
Assuming the lemma, we consider the phase of the second term in the last line of the equation \eqref{holonomy_symbol}. Lemma \ref{lemma_piece_integral} yields 
\begin{equation*}
        \frac{1}{2\pi}\sum_{l=1}^{m} \sum_{j=1,j\neq {k_l}}^n\alpha_j\left(\phi_j(p_{l+1})-\phi_j(p_{l})\right)
        =\sum_{l=1}^{m} \sum_{j=1,j\neq {k_l}}^n\alpha_j \cdot \re \left(\frac{1}{2\pi i}\int_{\gamma_l}\frac{1}{z-z_j}dz\right).
\end{equation*}
Rearranging the summation on the RHS yields
\begin{equation}
\label{frac_integral}
    \sum_{l=1}^{m} \sum_{j=1,j\neq {k_l}}^n\alpha_j \cdot \re \left(\frac{1}{2\pi i}\int_{\gamma_l}\frac{1}{z-z_j}dz\right)
    = \sum_{k=1}^n \alpha_k \cdot \re \left(\frac{1}{2\pi i}\int_{\gamma-\sum_{s_k\in\gamma_l}\gamma_l}\frac{1}{z-z_k}dz\right).
\end{equation}
Note that for $\gamma_l$ with $s_k\in\gamma_l$, indeed we have 
\begin{equation*}
    \re \left(\text{p.v. }\frac{1}{2\pi i}\int_{\gamma_l}\frac{1}{z-z_k}dz\right)=0.
\end{equation*}
Therefore, the RHS of the equation \eqref{frac_integral} equals to 
\begin{equation*}
    \sum_{k=1}^n \alpha_k \cdot \re \left(\text{p.v. }\frac{1}{2\pi i}\oint_{\gamma}\frac{1}{z-z_k}dz\right)
    = \sum_{k=1}^n \alpha_k \cdot \left(\text{p.v. }\frac{1}{2\pi i}\oint_{\gamma}\frac{1}{z-z_k}dz\right) 
    = \sum_{k=1}^n \alpha_k \cdot w_{\gamma, s_k}
\end{equation*}
by definition of the fractional winding number. This conclude the proof of Theorem \ref{thm1.1} in the introduction. It remains to show Lemma \ref{lemma_piece_integral}. 
\begin{proof}[Proof of Lemma \ref{lemma_piece_integral}]
It simply follows from the complex integration
\begin{equation*}
    \frac{1}{2\pi i}\int_{\gamma_l}\frac{1}{z-z_j}dz=\frac{1}{2\pi i}\log(z-z_j)\rvert_{p_l}^{p_{l+1}}
\end{equation*}
where the real part of the RHS equals to $\frac{1}{2\pi}\left(\phi_j(p_{l+1})-\phi_j(p_{l})\right)$. 
\end{proof}


\section{An application to lower bounds of resonances near the real axis}
In this section, we apply Theorem \ref{main_thm} to obtain a lower bound of number of scattering resonances in a logarithmic neighborhood of the positive real axis. We refer the reader to the comprehensive book on scattering resonances by Dyatlov and Zworski \cite{dyatlov2019mathematical} for the backgrounds and technical details regarding resonances that are not directly related to our application. The resonances result presented in this paper is closely related to the results of Galkowski \cite{galkowski2017quantitative} and Hillairet--Wunsch \cite{hillairet2017resonances} for manifolds with conic singularities, although our resonance lower bound employs a more general framework of Sj\"ostand \cite{sjostrand1997trace} on black box scattering with long range potentials. We further assume that \emph{there is only one closed diffractive geodesic for each possible length $L$} to prevent the possible cancellations between singularities of the wave trace at the same location $t=L$.

We briefly discuss and define resonances through a method of complex scaling following \cite[Section 5]{sjostrand1997trace}. For a more comprehensive discussion, we refer to \cite{sjostrand1997trace}. We first introduce a smooth submanifold $\Gamma\subset \mathbb{C}^2$ to which we shall deform the operator domain from $\RR^2$. Recall that $R$ is defined in Section \ref{reg&part} such that $\supp \chi \subset B(0,R)$. For given $\epsilon_0>0$ and $R_1>R$, we can construct a smooth function 
\begin{equation*}
    f_{\theta}(t): [0,\pi/2)\times [0,+\infty)\ni (\theta,t)\longmapsto f_{\theta}(t) \in \mathbb{C}
\end{equation*}
injective for every $\theta$, with the following properties,
\begin{enumerate}
  \item $f_{\theta}(t)=t$ for $0\leq t\leq R_1$;
  \item $0\leq \arg f_{\theta}(t) \leq \theta$, $\partial_t f_{\theta}(t)\neq 0$;
  \item $\arg f_{\theta}(t)\leq \arg \partial_t f_{\theta}(t)\leq \arg f_{\theta}(t)+\epsilon_0$;
  \item $f_{\theta}(t)=e^{i\theta}t$ for $t\geq T_0$, where $T_0$ only depends on $\epsilon_0$ and $R_1$. 
\end{enumerate}
We then define $\Gamma$ to be the image of the map $$\kappa_{\theta}: \RR^2 \ni x=tw\longmapsto f_{\theta}(t)w\in \mathbb{C}^2, t=|x|,$$
which coincide with $\RR^2$ in $B(0,R_1)$. 
In order to define the resonances, we use almost analytic extension to extend the operator $P_{\aalpha}$ from $\RR^2$ to an open neighborhood $\Omega$ of $\Gamma$ in $\mathbb{C}^2$. By restricting this extension to $\Gamma$, we can \emph{complex scale} the operator $P_{\aalpha}$ to 
$$P_{\aalpha,\Gamma}u\df (P_{\aalpha}\tilde{u})\rvert_{\Gamma} \text{ with } P_{\aalpha,\Gamma}: \CI(\Gamma)\rightarrow \CI(\Gamma)$$ 
where $\tilde{u}$ is the almost analytic extension of $u$. By \cite[Lemma 5.1, Lemma 5.2]{sjostrand1997trace} and the analytic Fredholm theory, the spectrum of $P_{\aalpha,\Gamma}$ in the set $e^{-2i[0,\theta)}[0,+\infty)$ is discrete and independent of the angle of scaling $\theta$. In other words, if we complex scale the operator $P_{\aalpha}$ by angles $\theta_1,\theta_2$ such that  $0<\theta_1<\theta_2<\frac{\pi}{2}$, then the spectrum of $P_{\aalpha,\Gamma_1}$ agrees with the spectrum of $P_{\aalpha,\Gamma_2}$ with the same multiplicities in the region $e^{-2i[0,\theta_1)}[0,+\infty)$. We therefore define $\mu_j$ to be a \emph{resonance} of $P_{\aalpha}$ (with multiplicity $k$) if $\mu_j^2$ is an eigenvalue of the scaled operator $P_{\aalpha,\Gamma}$ within the set $e^{-2i[0,\theta)}[0,+\infty)$ (with the same multiplicity $k$). We use $\text{Res}(P_{\aalpha})$ to denote the set of resonances of $P_{\aalpha}$.

Consider the regularized trace of the cosine propagator $\cos(t\sqrt{P_{\aalpha}})$. We define the regularized trace 
\begin{equation}
    u(t)\df \Tr\big[2\cos({it\sqrt{P_{\bullet}}})\restrictedto_{0}^{\aalpha}\big]\df \Tr\big[e^{it\sqrt{P_{\bullet}}}\restrictedto_{0}^{\aalpha}\big] + \Tr\big[e^{-it\sqrt{P_{\bullet}}}\restrictedto_{0}^{\aalpha}\big]\in \dom'(\RR).
\end{equation}
Therefore, $u(t)$ is in the trace class in the sense of distribution following the discussion in Section \ref{reg&part}. Moreover, for any $\rho\in\CI_c(\RR)$, we have 
\begin{equation}
    \la u,\rho\ra = \Tr \big[ \rho(t)\big(e^{it\sqrt{P_{\bullet}}}+e^{-it\sqrt{P_{\bullet}}}\big)\restrictedto_{0}^{\aalpha} \big].
\end{equation}
Note that $\rho(t)\big(e^{it\sqrt{P_{\bullet}}}+e^{-it\sqrt{P_{\bullet}}}\big)=\tilde{\rho}(P_{\bullet})$, where $\tilde{\rho}(z)=\hat{\rho}(\sqrt{z})+\hat{\rho}(-\sqrt{z})$ is an entire function in $z$ ($\hat{\rho}$ is the Fourier transform of $\rho$).

In the previous section we have obtained singularities of the trace of the (half-)wave propagator $e^{it\sqrt{P_{\aalpha}}}$; we shall derive a similar result for $u(t)$ now. We observe that the principal amplitude of the regularized trace of $e^{-it\sqrt{P_{\aalpha}}}$ will not cancel off the principal amplitude of $e^{it\sqrt{P_{\aalpha}}}$ near the conormal singularity at $\{t=L\}$ for any $L>0$ being the length of a closed diffractive geodesic. In order to see this, by the unitarity of the half-wave propagator $U_{\aalpha}(-t)=U_{\aalpha}(t)^*$, the Schwartz kernel $K(t,z,z')$ of the propagator $e^{it\sqrt{P_{\aalpha}}}$ satisfies 
$$K(-t,z,z')=\overline{K(t,z',z)}.$$
Let $K_0(t,z,z')$ denote the kernel of the free propagator $e^{it\sqrt{P_{0}}}$ in the regularization.  
Therefore, taking the (regularized) trace yields 
$$\int \big[K-K_0\big](-t,z,z) dz\equiv \int \overline{\big[K-K_0\big](t,z,z)} dz \mod \CI.$$
Written in terms of the oscillatory integral with phase function $\phi=\lambda(t-L)$ near the singularities at $t=L$, this amounts to taking the symbol map:
$$a(z;\lambda)\mapsto \overline{a(z;-\lambda)}.$$
Note that the principal symbol of $e^{it\sqrt{P_{\aalpha}}}$ is supported in $\RR_+$ for the phase variable $\lambda$ (cf. Proposition \ref{prop_pgtr1}). Switching the sign of the phase variable $\lambda$ yields that the principal symbol of $e^{-it\sqrt{P_{\aalpha}}}$ is of the same order but is now supported in $\RR_-$ for the phase variable $\lambda$. Therefore, $u(t)$ must have conormal singularities of the same order at the same locations as $\Tr\big[e^{it\sqrt{P_{\bullet}}}\restrictedto_{0}^{\aalpha}\big]$. Note that this can also be verified directly using \cite[Theorem 6.1]{yang2020ab} and the functional calculus with the microlocality of $\sqrt{P_{\aalpha}}$ proved in the Proposition \ref{ap_mlty} of the Appendix. 

Following Sj\"ostrand \cite[Section 10]{sjostrand1997trace}, let $$\Lambda\df \text{Res}(P_{\aalpha}) \cap \{z\in\mathbb{C}; \ \lvert z\rvert\geq C, 0\leq -\arg z\leq 1/C\}$$ 
for some $C>1$. For $\nu>0$, we define
\begin{equation*}
    \Lambda_{\nu}\df\big\{ \mu_j\in\Lambda\restrictedto -\im\mu_j\leq\nu \log\lvert\mu_j\rvert\big\}
\end{equation*}
to be the set of resonances in a logarithmic neighborhood below the positive real axis, and
\begin{equation*}
    N_{\nu}(r)\df \#\big\{ \mu_j\in\Lambda_{\nu}\restrictedto \re \mu_j\leq r\big\}
\end{equation*}
to denote the counting function of the number of resonances $\mu_j$ with $\re \mu_j \leq r$ in the set $\Lambda_{\nu}$. Note that the ``free" Hamiltonian $P_{0}$ with the vector potential $\Vec{A}_{\tilde{\alpha}}$ used in the regularization has no resonance (at least away from zero resonances) by a similar argument to \cite[Theorem 1.2]{baskin2020scattering} (see also \cite[Section 2]{alexandrova2014resonances}).

Now we state the following theorem of Sj\"ostrand \cite[Theorem 10.1]{sjostrand1997trace} which relates the conormal singularities of the wave trace $u(t)$ to a lower bound of the number of resonances in the region $\Lambda_{\nu}$. 
\begin{theorem}[Sj\"ostrand]
\label{thm_sjo}
Let $k<0,\ L,b>0$ and suppose that for all $\rho\in\CI_c\big((0,+\infty)\big)$ supported in a sufficient small neighborhood of $L$ with $\rho(L)=1$,
\begin{equation*}
    \lvert \widehat{\rho u}(\lambda)\rvert \geq \big(b-o(1)\big)\cdot\lambda^k, \ \lambda\longrightarrow\infty.
\end{equation*}
Then for some $n_1>0$ and every $\epsilon>0$, $\nu>\frac{n_1-k}{L-\epsilon}$ and $\delta>0$, there exists $r(\delta)>0$, such that 
$$N_{\nu}(r)\geq r^{1-\delta}$$
for $r>r(\delta)$.
\end{theorem}

It therefore remains to combine Theorem \ref{main_thm} and Theorem \ref{thm_sjo} to obtain a lower bound on the number of resonances in a logarithmic neighborhood.

\begin{corollary}
\label{cor_res}
If these exists a closed diffractive geodesic of length $L$ with $m$ diffractions, then for some $n_1$ and every $\epsilon>0$, $\nu>\frac{n_1+m/2}{L-\epsilon}$ and $\delta>0$, there exists $r(\delta)>0$, such that 
\begin{equation}
\label{lb}
    N_{\nu}(r)\geq r^{1-\delta}
\end{equation}
for $r>r(\delta)$. In particular, if $d_{\text{max}}\df\max\big\{d(s_i,s_j)\restrictedto\ s_i,s_j\in S\big\}$, then 
\begin{equation}
\label{lb2}
    \#\left\{\mu_j\bigg\rvert \ 0\leq -\ima \mu_j\leq \left(\frac{1}{2d_{\text{max}}}+\epsilon\right)\cdot\log |\mu_j| ,\ \rea\mu_j\leq r\right\}\geq r^{1-\delta}.
\end{equation}
the smallest logarithmic neighborhood below the positive real axis with this lower bound.
\end{corollary}

\begin{remark}
For such a lower bound, equation \eqref{lb2} gives the smallest logarithmic neighborhood below the positive real axis obtained using our trace formula. This neighborhood is given by the longest closed diffractive geodesic with the smallest number of diffractions, which corresponds to the closed diffractive geodesics bouncing between the farthest two solenoids. 
\end{remark}

\begin{proof}[Proof of Corollary \ref{cor_res}]
Consider a closed diffractive geodesic of length $L$ with $m$ diffractions. By Theorem \ref{main_thm}, the principal part of the singularity of $\rho u(t)$ at $t=L$ is of order $\mathcal{O}(\lambda^{-\frac{m}{2}})$ as $\lambda\rightarrow\infty$. Applying Theorem \ref{thm_sjo} yields the corresponding lower bound. For the lower bound in equation \eqref{lb2}, take the closed diffractive geodesic with two diffractions at $s_i,s_j$ such that $d(s_j,s_j)=d_{\text{max}}$. Consider this closed diffractive geodesic bounces between the two solenoids for $k$ rounds such that $L=2kd_{\text{max}}$. It is then being diffracted $2k$ times. Letting $k\rightarrow\infty$ yields the desired bound.
\end{proof}

\appendix
\section{Microlocality of the square root operator}
The goal of this appendix is to show the following proposition on microlocality of $\sqrt{P_{\aalpha}}$. The exposition of this appendix is inspired by the discussion of microlocality of $\sqrt{\Lap}$ on manifolds with conic singularities in \cite{hillairet2017resonances}. 

\begin{proposition}
\label{ap_mlty}
On $X=\RRS$, consider the Friedrichs extension operator $P_{\aalpha}$. Then
\begin{enumerate}
    \item For any open sets $U, V\in \RR^2$ such that $U\cap V=\emptyset$, and $V\subset\RRS$, for any $N$, $P_{\aalpha}^N\sqrt{P_{\aalpha}}$ is continuous from $L^2(V)$ into $L^2(U)$. 
    \item For any simply connected open set $U$ such that $\Bar{U}\subset \RRS$, $\sqrt{P_{\aalpha}}$ is a pseudodifferential operator from $H^1_c(U)$ to $L^2_{loc}(U)$. 
\end{enumerate}
\end{proposition}

\begin{proof}
Using functional calculus, we write $\sqrt{P_{\aalpha}}$ in terms of heat kernel: 
\begin{equation}
    \sqrt{P_{\aalpha}}=\frac{P_{\aalpha}}{\Gamma(\frac{1}{2})}\int_0^{\infty}e^{-tP_{\aalpha}}t^{-\frac{1}{2}}dt.
\end{equation}
Take a smooth cutoff function $\rho\in\CI_c([0,+\infty))$ such that $\rho\equiv 1$ on $[0,2t_0]$ for some $t_0>0$, then take $\psi=1-\rho$. Since $e^{-tP_{\aalpha}}$:  $\dom_{\aalpha}^{-\infty}\mapsto \bigcap_{n} \dom_{\aalpha}^{n}$ and 
\begin{equation}
    \int_0^{\infty}e^{-tP_{\aalpha}}\psi(t)t^{-\frac{1}{2}}dt=e^{-t_0P_{\aalpha}}\int_0^{\infty}e^{-(t-t_0)P_{\aalpha}}\psi(t)t^{-\frac{1}{2}}dt,
\end{equation}
the operator defined in the left-hand side of the equation is smoothing. Therefore, we only need to consider the operator
\begin{equation}
    P_{\aalpha} \int_0^{\infty}e^{-tP_{\aalpha}}\rho(t)t^{-\frac{1}{2}}dt.
\end{equation}
Since $P_{\aalpha}$ is a differential operator, it suffices to consider the microlocality of the operator-valued integral: 
\begin{equation}
\label{red}
    \int_0^{\infty}e^{-tP_{\aalpha}}\rho(t)t^{-\frac{1}{2}}dt.
\end{equation}
We first prove (1). Take two open sets $U, V\in \RR^2$ such that $U\cap V=\emptyset$, and $V\subset\RRS$. By the reduction above, we need to show 
\begin{equation}
\label{red_1}
    P_{\aalpha}^N\int_0^{\infty}e^{-tP_{\aalpha}}\rho(t)t^{-\frac{1}{2}}dt:\ L^2(V)\longrightarrow L^2(U)
\end{equation}
is continuous for any $N$. Therefore we consider the distribution $T_a$ on $\RR\times V$ defined by
\begin{equation}
    \la T_a, \varphi(t)b(y)\ra_{\dom'\times\dom}\df\int_0^{\infty}\la a,e^{-tP_{\aalpha}}b\ra_{L^2}\varphi(t)dt.
\end{equation}
In the sense of distributions in $\dom'(\RR\times V)$, 
\begin{equation}
    (\partial_t+P_{\aalpha})T_a=0.
\end{equation}
Thus $T_a\in\CI(\RR\times V)$ by hypoellipticity (cf. \cite[Chapter I.5]{shubin1987pseudodifferential}) of $\partial_t+P_{\aalpha}$ on $\RR\times V$. Since $T_a$ vanishes for $t<0$ by definition, for any $(a,b)\in L^2(U)\times L^2(V)$, the distribution
\begin{equation}
    t\mapsto\la e^{-tP_{\aalpha}}a,b\ra_{L^2}
\end{equation}
is smooth on $[0,+\infty)$ and vanishes at $0$ to infinite order. Therefore the $N$-th derivative of the RHS vanishes of order $k$ at $t=0$ for any $N$ and $k$. Thus
\begin{equation}
    t^{-k}\la P_{\aalpha}^N e^{-tP_{\aalpha}}a,b\ra_{L^2}
\end{equation}
is bounded in $(0,1]$. The Principle of Uniform Boundedness thus yields 
\begin{equation}
    \| P_{\aalpha}^N e^{-t\pa}\|_{L^2(V)\rightarrow  L^2(U)}=\mathcal{O}(t^k) \text{ as } t\rightarrow 0
\end{equation}
for any $N,k$. Thus, the operator \eqref{red_1} is continuous for any $N$ and we proved (1). 

Now we prove (2) by comparison with the heat kernel of the standard Laplacian $\Lap$ on $\RR^2$. We only need to show \eqref{red} is a pseudodifferential operator. Choose a simply connected open set $U$ such that $U\ssubset \RRS$. Let $e$ denote the heat kernel of $P_{\aalpha}$ on $\RRS$ and $\tilde{e}$ denote the heat kernel of $e^{-i\phi} \Lap e^{i\phi}$ in $U$, where $\Lap$ is the standard Laplacian and $\phi=\sum_j \alpha_j \phi_j\in \CI(U)$ is defined in the equation \eqref{ang_i}. Note that in particular $P_{\aalpha}$ and $\Lap$ are gauge equivalent in $U$, so are the corresponding heat kernels. Furthermore, we have the operator identity
$$e^{-i\phi} \Lap e^{i\phi}=\Lap + e^{-i\phi}[\Lap, e^{i\phi}]$$
locally in $U$, therefore $\WF' e^{-i\phi} \Lap e^{i\phi}= \WF' \Lap$.  
Consider $r\df e-\tilde{e}$ as a distribution on $\RR\times U\times U$ defined by
\begin{equation}
    \la r, \phi \ra=\int_0^{\infty}\int_{U\times U}\left(e(t,x,y)-\tilde{e}(t,x,y)\right)\phi(t,x,y)dxdydt.
\end{equation}
As before, in the sense of distributions
\begin{equation}
    (2\partial_t+P_{\aalpha,x}+P_{\aalpha,y})r=0 \text{ on } \RR\times U\times U.
\end{equation}
So by hypoellipticity on $\RR\times U\times U$, $r\in\CI(\RR\times U\times U)$. Consider the Schwartz kernel
\begin{equation}
\label{red_2}
    \int_0^{\infty}r(t,x,y)\rho(t)t^{-\frac{1}{2}}dt.
\end{equation}
As in the proof of (1), $r$ is smooth and vanishes at $t=0$ to infinite order. The Schwartz kernel \eqref{red_2} therefore defines a smoothing operator. Thus, \eqref{red} is a pseudodifferential operator follows from it being the sum of \eqref{red_2} and
$$\int_0^{\infty}\tilde{e}(t,x,y)\rho(t)t^{-\frac{1}{2}}dt,$$
which is a pseudodifferential operator due to the functional calculus of $\Lap$ (or equivalently $e^{-i\phi} \Lap e^{i\phi}$ near $U$) on $\RR^2$. 
\end{proof}

\bibliography{reference}

\begin{thebibliography}{BGR82}

\bibitem[AB59]{aharonov1959significance}
Yakir Aharonov and David Bohm.
\newblock Significance of electromagnetic potentials in the quantum theory.
\newblock {\em Physical Review}, 115(3):485, 1959.

\bibitem[AT98]{adami1998aharonov}
R~Adami and A~Teta.
\newblock On the {A}haronov--{B}ohm {H}amiltonian.
\newblock {\em Letters in Mathematical Physics}, 43(1):43--54, 1998.

\bibitem[AT11]{alexandrova2011resonance}
Ivana Alexandrova and Hideo Tamura.
\newblock Resonance free regions in magnetic scattering by two solenoidal
  fields at large separation.
\newblock {\em Journal of Functional Analysis}, 260(6):1836--1885, 2011.

\bibitem[AT14]{alexandrova2014resonances}
Ivana Alexandrova and Hideo Tamura.
\newblock Resonances in scattering by two magnetic fields at large separation
  and a complex scaling method.
\newblock {\em Advances in Mathematics}, 256:398--448, 2014.

\bibitem[BGR82]{bardos1982relation}
Claude Bardos, Jean-Claude Guillot, and James Ralston.
\newblock La relation de {P}oisson pour l'équation des ondes dans un ouvert
  non borné application a la theorie de la diffusion.
\newblock {\em Communications in Partial Differential Equations},
  7(8):905--958, 1982.

\bibitem[BPS00]{bogomolny2000diffractive}
Eugene Bogomolny, Nicolas Pavloff, and Charles Schmit.
\newblock Diffractive corrections in the trace formula for polygonal billiards.
\newblock {\em Physical Review E}, 61(4):3689, 2000.

\bibitem[BY20]{baskin2020scattering}
Dean Baskin and Mengxuan Yang.
\newblock Scattering resonances on truncated cones.
\newblock {\em Pure and Applied Analysis}, 2(2):385--396, 2020.

\bibitem[DG75]{duistermaat1975spectrum}
JJ~Duistermaat and VW~Guillemin.
\newblock The spectrum of positive elliptic operators and periodic
  bicharacteristics.
\newblock {\em Inventiones Mathematicae}, 29:39--80, 1975.

\bibitem[DH72]{duistermaat1972fourier}
Johannes~Jisse Duistermaat and Lars H{\"o}rmander.
\newblock Fourier integral operators. ii.
\newblock {\em Acta mathematica}, 128(1):183--269, 1972.

\bibitem[D{\v{S}}98]{stovicek1998aharonov}
Ludwik Dabrowski and P.~{\v{S}}t'ov{\i}{\v{c}}ek.
\newblock {A}haronov--{B}ohm effect with $\delta$-type interaction.
\newblock {\em Journal of Mathematical Physics}, 39(1):47--62, 1998.

\bibitem[DZ19]{dyatlov2019mathematical}
Semyon Dyatlov and Maciej Zworski.
\newblock {\em Mathematical theory of scattering resonances}, volume 200.
\newblock American Mathematical Soc., 2019.

\bibitem[ER14]{eskin2014aharonov}
Gregory Eskin and James Ralston.
\newblock The {A}haronov--{B}ohm effect in spectral asymptotics of the magnetic
  schr{\"o}dinger operator.
\newblock {\em Analysis \& PDE}, 7(1):245--266, 2014.

\bibitem[E{\v{S}}V02]{stovicek2002generalized}
Pavel Exner, P~{\v{S}}t'ov{\'\i}{\v{c}}ek, and P~Vyt{\v{r}}as.
\newblock Generalized boundary conditions for the {A}haronov--{B}ohm effect
  combined with a homogeneous magnetic field.
\newblock {\em Journal of Mathematical Physics}, 43(5):2151--2168, 2002.

\bibitem[FW17]{ford2017diffractive}
G~Austin Ford and Jared Wunsch.
\newblock The diffractive wave trace on manifolds with conic singularities.
\newblock {\em Advances in Mathematics}, 304:1330--1385, 2017.

\bibitem[Gal17]{galkowski2017quantitative}
Jeffrey Galkowski.
\newblock A quantitative {V}ainberg method for black box scattering.
\newblock {\em Communications in Mathematical Physics}, 349(2):527--549, 2017.

\bibitem[Hil05]{hillairet2005contribution}
Luc Hillairet.
\newblock Contribution of periodic diffractive geodesics.
\newblock {\em Journal of Functional Analysis}, 226(1):48--89, 2005.

\bibitem[H{\"o}r71]{hormander1971fourier}
Lars H{\"o}rmander.
\newblock Fourier integral operators. i.
\newblock {\em Acta mathematica}, 127(1):79, 1971.

\bibitem[HW17]{hillairet2017resonances}
Luc Hillairet and Jared Wunsch.
\newblock On resonances generated by conic diffraction.
\newblock {\em arXiv preprint arXiv:1706.07869}, 2017.

\bibitem[HW19]{hungerbuhler2019non}
Norbert Hungerb{\"u}hler and Micha Wasem.
\newblock Non-integer valued winding numbers and a generalized residue theorem.
\newblock {\em Journal of Mathematics}, 2019.

\bibitem[IT01]{ito2001aharonov}
Hiroshi~T Ito and Hideo Tamura.
\newblock {A}haronov-{B}ohm effect in scattering by point-like magnetic fields
  at large separation.
\newblock {\em Annales Henri Poincar{\'e}}, 2(2):309--359, 2001.

\bibitem[IT06]{ito2006semiclassical}
Hiroshi~T Ito and Hideo Tamura.
\newblock Semiclassical analysis for magnetic scattering by two solenoidal
  fields.
\newblock {\em Journal of the London Mathematical Society}, 74(3):695--716,
  2006.

\bibitem[Mel82]{melrose1982scattering}
Richard Melrose.
\newblock Scattering theory and the trace of the wave group.
\newblock {\em Journal of Functional Analysis}, 45(1):29--40, 1982.

\bibitem[Min05]{mine2005aharonov}
Takuya Mine.
\newblock The {A}haronov-{B}ohm solenoids in a constant magnetic field.
\newblock {\em Annales Henri Poincar{\'e}}, 6(1):125--154, 2005.

\bibitem[Shu87]{shubin1987pseudodifferential}
Mikhail~Aleksandrovich Shubin.
\newblock {\em Pseudodifferential operators and spectral theory}, volume 200.
\newblock Springer, 1987.

\bibitem[Sj{\"o}97]{sjostrand1997trace}
Johannes Sj{\"o}strand.
\newblock A trace formula and review of some estimates for resonances.
\newblock In {\em Microlocal analysis and spectral theory}, pages 377--437.
  Springer, 1997.

\bibitem[{\v{S}}t'89]{stovicek1989green}
Pavel {\v{S}}t'ov{\'\i}{\v{c}}ek.
\newblock The {G}reen function for the two-solenoid {A}haronov-{B}ohm effect.
\newblock {\em Physics Letters,(Section) A}, 142(1):5--10, 1989.

\bibitem[{\v{S}}t'91]{vstovivcek1991krein}
Pavel {\v{S}}t'ov{\'\i}{\v{c}}ek.
\newblock Krein’s formula approach to the multisolenoid {A}haronov-{B}ohm
  effect.
\newblock {\em Journal of mathematical physics}, 32(8):2114--2122, 1991.

\bibitem[Tam07]{tamura2007semiclassical}
Hideo Tamura.
\newblock Semiclassical analysis for magnetic scattering by two solenoidal
  fields: total cross sections.
\newblock {\em Annales Henri Poincar{\'e}}, 8(6):1071--1114, 2007.

\bibitem[Tam08]{tamura2008time}
Hideo Tamura.
\newblock Time delay in scattering by potentials and by magnetic fields with
  two supports at large separation.
\newblock {\em Journal of Functional Analysis}, 254(7):1735--1775, 2008.

\bibitem[Tam17]{tamura2017aharonov}
Hideo Tamura.
\newblock {A}haronov--{B}ohm effect in resonances for scattering by three
  solenoids at large separation.
\newblock {\em Applied Mathematics Research eXpress}, 2017(1):65--117, 2017.

\bibitem[Yan21]{yang2020ab}
Mengxuan Yang.
\newblock Diffraction of the {A}haronov-{B}ohm {H}amiltonian.
\newblock {\em Annales Henri Poincaré}, 2021.

\bibitem[Zwo12]{zworski2012semiclassical}
Maciej Zworski.
\newblock {\em Semiclassical analysis}, volume 138.
\newblock American Mathematical Soc., 2012.

\end{thebibliography}
\bibliographystyle{alpha}

\end{document}